\documentclass{article}
\pdfoutput=1
\usepackage{stdpkg-minimal}
\usepackage{functional-analysis}
\usepackage{group-theory}
\usepackage[retainorgcmds]{IEEEtrantools}
\usepackage{tikz}
\usepackage{scalerel}

\theoremstyle{plain}
\newtheorem{theoremletter}{Theorem}

\crefname{theoremletter}{Theorem}{Theorems}
\theoremstyle{definition}
\newtheorem{exampleletter}[theoremletter]{Example}

\crefname{exampleletter}{Example}{Examples}

\theoremstyle{definition}
\newtheorem{definitionletter}[theoremletter]{Definition}

\usepackage[backend=biber,style=vaes]{biblatex}
\newcommand{\inter}{\operatorname{int}}
\ProvideMathOperator{\covol}{covol}
\newcommand{\otmin}{\otimes_{\text{\rmfamily min}}}
\DeclareMathOperator*{\freeprod}{{\scaleobj{1.5}{\ast}}}

\title{Ozawa's class \texorpdfstring{$\mathcal S$}{𝓢} for locally compact groups and unique prime factorization of group von Neumann algebras}
\author{Tobe Deprez\thanks{KU~Leuven, Department of Mathematics, Leuven (Belgium), tobe.deprez@kuleuven.be.
        Supported by a PhD fellowship of the Research Foundation Flanders (FWO). Part of this research was performed while the author was visiting the Institute for Pure and Applied Mathematics (IPAM), which is supported by the National Science Foundation.}}
\date{\today}

\begin{document}
    \maketitle
    \begin{abstract}
        \noindent We study class $\mathcal S$ for locally compact groups. We characterize locally compact 
        groups in
        this class as groups having an amenable action on a boundary that is small at infinity, 
        generalizing a theorem of \citeauthor{Ozawa2006}.
        Using this characterization, we provide new examples of groups in class $\mathcal S$ and prove unique 
        prime factorization results for group von Neumann 
        algebras of products of locally compact groups in this class. We also prove that class $\mathcal S$ is a measure
        equivalence invariant.
    \end{abstract}

    \section{Introduction}
    Class $\mathcal S$ for countable groups was introduced by \citeauthor{Ozawa2006} in 
    \cite{Ozawa2006}. A countable group $\Gamma$ is said to be in \emph{class $\mathcal S$} if it is exact and
    it admits a map $\eta: \Gamma \to \Prob(\Gamma)$ satisfying
    \[\lim_{k\to \infty} \norm{\eta(gkh) - g\cdot \eta(k)}= 0\]
    for all $g,h\in G$. Equivalently, class $\mathcal S$ can be characterized as the class of all groups that 
    admit an amenable action on a boundary that is \emph{small at infinity} (see \cite[Theorem~4.1]{Ozawa2006}). 
    Groups in class $\mathcal S$ are also called \imph{bi-exact}.
    
    Class $\mathcal S$ is used in, among others, 
    \cite{Ozawa2004solid,Ozawa2006,OP2004prime,Chifan2011,Popa2014a,CI2018amalgamated,CdSS2016products} to prove 
    rigidity results for group von Neumann algebras of countable groups. In \cite{Ozawa2004solid}, 
    \citeauthor{Ozawa2004solid} proved that the group von Neumann algebra
    $L(\Gamma)$ is solid when $\Gamma$ belongs to class $\mathcal S$. This implies in particular that for $\Gamma$ icc, 
    nonamenable and in class $\mathcal S$, the group von Neumann algebra $L(\Gamma)$ is 
    \emph{prime}, i.e. $L(\Gamma)$ does not
    decompose as a tensor product $M_1\otimes M_2$ for non-type I factors $M_1$ and $M_2$.
    
    In \cite{OP2004prime}, \citeauthor{OP2004prime} 
    proved the first unique prime factorization results for von Neumann algebras using groups in this class. 
    Among other results, they showed that if $\Gamma = \Gamma_1\times \dots\times \Gamma_n$ is a product of 
    nonamenable, icc groups in class $\mathcal S$, then
    $L(\Gamma)\cong L(\Gamma_1)\otimes \dots \otimes L(\Gamma_n)$ 
    remembers the number of factors $n$ and each factor $L(\Gamma_i)$ up to amplification, i.e. if 
    $L(\Gamma)\cong N_1\otimes \dots\otimes N_m$ for some prime factors $N_1,\dots,N_m$, then $n=m$ and (after 
    relabeling) $L(\Gamma_i)$ is stably isomorphic to $N_i$ for $i=1,\dots,n$.   
    Subclasses of class $\mathcal S$ were used in 
    \cite{Chifan2011,Popa2014a,Houdayer2013} to prove 
    rigidity results on crossed product von Neumann algebras $L^\infty(X)\rtimes \Gamma$.
    
    Examples of countable groups in class $\mathcal S$ are amenable groups, hyperbolic groups (see 
    \cite{Adams1994boundary}), 
    lattices in connected simple Lie groups of rank one (see \cite[Proof of Théorème 4.4]{Skandalis1988}), wreath 
    products $B\wr \Gamma$ with $B$ amenable and $\Gamma$ in class $\mathcal S$ (see \cite{Ozawa2006}) and 
    $\mathbb Z^2\rtimes \SL_2(\mathbb Z)$ (see \cite{Ozawa2009}). Moreover, class $\mathcal S$ is closed 
    under measure equivalence (see 
    \cite{Sako2009}). Examples of groups not belonging to class $\mathcal S$ are product groups 
    $\Gamma\times \Lambda$ with $\Gamma$ nonamenable and $\Lambda$ infinite, nonamenable inner amenable groups and 
    nonamenable groups with infinite center.
    
    In this paper, we study class $\mathcal S$ for locally compact groups. We provide a characterization of groups 
    in 
    this class similar to \cite[Theorem~4.1]{Ozawa2006}, we provide new examples 
    of groups 
    in this class and we 
    prove a unique prime factorization result for group von Neumann algebras of locally compact groups. We also prove 
    that class $\mathcal S$ is a measure equivalence invariant.
    
    Let $G$ be a locally compact second countable (lcsc) group. We denote by $\Prob(G)$ the space of all Borel 
    probability measures, i.e.
    the state space of $C_0(G)$. The precise definition of class $\mathcal S$ for locally compact groups is now 
    as follows.
    \begin{definitionletter}
        Let $G$ be a lcsc group. We say that $G$ is in 
        \imph{class $\mathcal S$} (or \imph{bi-exact}) if 
        $G$ is exact and if there exists a $\norm{.}$-continuous map $\eta: G\to \Prob(G)$ satisfying
        \begin{align}
            \lim_{k\to \infty}\norm{\eta(gkh) - g\cdot \eta(k)}= 0\label{eq:prop (S)}
        \end{align}
        uniformly on compact sets for $g,h\in G$.
    \end{definitionletter}
    In \cite{Brothier2017rigidity} this property without the exactness condition was called \imph{property (S)}. Note 
    that the definition was slightly different: the image of the map $\eta$ above was in the space 
    $\mathcal S(G) = \set{f\in L^1(G)^+}{\norm{f}_1=1}$ instead of $\Prob(G)$. However, we prove in 
    \cref{prop(S).equivalent}
    that this is equivalent. It is also worthwhile to note that it is currently unknown whether there are groups 
    with property (S) that are not exact. 
    
    Examples of lcsc groups in class $\mathcal S$ include amenable groups, 
    groups acting continuously and properly on a tree or hyperbolic graph of uniformly bounded degree and connected, 
    simple Lie groups of real rank one with finite center. Proofs of these results can be found in 
    \cite[Section~7]{Brothier2017rigidity}.
    It is easy to prove that groups not in class $\mathcal S$ include 
    product groups $G\times H$ with $G$ nonamenable and $H$ non-compact, nonamenable groups $G$ with non-compact 
    center
    and nonamenable groups $G$ that are \emph{inner amenable at infinity}, i.e. for which there exists a conjugation 
    invariant mean $m$ on $G$ such that $m(E) = 0$ for every compact set $E\subseteq G$ (see also 
    \cref{bi-exact.inner amenable}).
    
    We prove a version of \cite[Theorem~4.1]{Ozawa2006} for locally compact groups in class $\mathcal S$, 
    i.e. we characterize groups in class $\mathcal S$ as groups acting amenably on a boundary that is 
    \emph{small at infinity}.
    Given a locally compact group $G$, we denote by $C_b^u(G)$ the algebra of \emph{bounded uniformly continuous} 
    functions on $G$, i.e. the bounded functions $f: G\to \mathbb C$ such that
    \[\norm{\lambda_gf - f}_\infty \to 0\mathand \norm{\rho_gf - f}_\infty \to 0\]
    whenever $g\to e$. Here, $\lambda$ and $\rho$ denote the left and right regular representations respectively, i.e.
    $(\lambda_gf)(h) = f(g^{-1}h)$ and $(\rho_gf)(h) = f(hg)$.
    We define the compactification $h^uG$ of the group $G$ as the spectrum of the following algebra
    \[C(h^uG) \cong \set{f\in C_b^u(G)}{\text{$\rho_gf - f\in C_0(G)$ for all $g\in G$}}\]
    and denote by $\nu^uG = h^uG\setminus G$ its boundary. The compactification $h^uG$ is equivariant in the sense that 
    both actions 
    $G\acts G$ by left and right translation extend to continuous actions $G\acts h^uG$. It is also 
    \emph{small at infinity} in the sense that the extension of the action by right translation is trivial on the 
    boundary $\nu^uG$. It is moreover the universal equivariant compactification that is small at infinity, in the 
    sense that for every equivariant compactification $\overline G$ that is small at infinity, the inclusion 
    $G\hookrightarrow \overline G$ extends to a continuous $G$-equivariant map $h^uG\to \overline G$.
    
    The characterization 
    of groups in class $\mathcal S$ now goes as follows.
    \begin{theoremletter}\label{bi-exact.equivalent}
        Let $G$ be a lcsc group. Then, the following are equivalent
        \begin{equivalent}
            \item $G$ is in class $\mathcal S$,\label{bi-exact.equivalent.bi-exact}
            \item the action $G\acts \nu^uG$ induced by left translation is topologically amenable,\label{bi-exact.equivalent.boundary}
            \item the action $G\acts h^uG$ induced by left translation is topologically amenable,\label{bi-exact.equivalent.compactification}
            \item the action $G\times G\acts C_b^u(G) / C_0(G)$ induced by left and right translation is topologically 
            amenable.\label{bi-exact.equivalent.GxG action}
        \end{equivalent}
    \end{theoremletter}
    The two novelties in the proof of this result are the proof of \ref{bi-exact.equivalent.compactification} and the 
    method we used to prove the implication 
    \ref{bi-exact.equivalent.GxG action}$\Rightarrow$\ref{bi-exact.equivalent.bi-exact}. Indeed, in the original
    proof of \citeauthor{Ozawa2006} in the countable setting, it was proven that $G$ belongs to class $\mathcal S$ 
    if 
    and only if there exists a u.c.p map 
    $\theta: \Cred(G)\otmin \Cred(G)\to B(L^2(G))$ satisfying 
    $\varphi(x\otimes y) - \lambda(x)\rho(y)\in K(L^2(G))$, where $\lambda$ and $\rho$ denote the representations of 
    $\Cred(G)$ induced by the left and right regular representation respectively.
    This is however no longer true for locally compact groups. Indeed, for all connected groups $G$, the reduced 
    \C*-algebra $\Cred(G)$ is nuclear and hence a map $\theta$ as above always exists.
    
    Denote by $\beta^{lu}G$ the left-equivariant Stone-Čech compactification of $G$, i.e. the spectrum of the algebra 
    $C_b^{lu}(G)$ of bounded left-uniformly continuous functions on $G$. The action $G\acts G$ by left-translation 
    extends uniquely to a continuous action $G\acts \beta^{lu}G$. Moreover, $\beta^{lu}G$ is the 
    universal left-equivariant compactification of $G$ in the sense that every left-$G$-equivariant continuous map 
    $G\to X$ to any compact space $X$ with continuous action $G\acts X$ extends uniquely to a $G$-equivariant 
    continuous map $\beta^{lu}G\to X$. We also prove the following characterization of groups in class $\mathcal S$.
    \begin{theoremletter}\label{bi-exact.map to Prob(beta^luG)}
        Let $G$ be a lcsc group. Then $G$ belongs to class $\mathcal S$ if and only if $G$ is exact and there exists a 
        Borel map $\eta: G\to \Prob(\beta^{lu}G)$ satisfying
        \[\lim_{k\to \infty} \norm{\eta(gkh) - g\cdot \eta(k)} = 0\]
        uniformly on compact sets for $g,h\in G$.
    \end{theoremletter}
    In the proof of this theorem, we will see that it is precisely the exactness of $G$ that allows us to construct the
    required map $\tilde \eta: G\to \Prob(G)$ from a map $\eta: G\to \Prob(\beta^{lu}G)$. This was implicitly observed 
    before in \cite[Chapter~15]{Brown2008} for countable groups.
    
    Using \cref{bi-exact.equivalent}, we prove two new examples of locally compact groups in class 
    $\mathcal S$. The first example is the following.
    \begin{theoremletter}\label{bi-exact.R^2 rtimes SL_2(R)}
        The group $\mathbb R^2\rtimes \SL_2(\mathbb R)$ belongs to class $\mathcal S$.
    \end{theoremletter}
    This result is a locally compact version of the main result in \cite{Ozawa2009}, where it was proven that 
    $\mathbb Z^2\rtimes \SL_2(\mathbb Z)$ belongs to class $\mathcal S$.

    In \cite{Cornulier2017wreath}, \citeauthor{Cornulier2017wreath} introduced a notion of wreath products for locally 
    compact groups. See \eqref{eq.wreath.definition} on page \pageref{eq.wreath.definition} for a short recapitulation 
    of 
    this notion and the notation used in 
    this article. The following result is a locally compact version of \cite[Corollary~4.5]{Ozawa2006}.
    \begin{theoremletter}\label{bi-exact.wreath product.introduction}
        Let $B$ and $H$ be lcsc groups, $X$ a countable set with a continuous action $H\acts X$ and $A\subseteq B$ be a 
        compact open subgroup. If $B$ is amenable, all stabilizers $\Stab_H(x)$ for $x\in X$ are amenable and $H$ 
        belongs to class $\mathcal S$, then also the wreath product $B\wr_X^A H$ belongs to class $\mathcal S$.
    \end{theoremletter}
    
    A notion of measure equivalence for locally compact groups was introduced by S. \citeauthor{Deprez2014permanence} in
    \cite{Deprez2014permanence}. By \cite[Corollary~2.9]{DL2015propA} and \cite[Theorem~0.1~(6)]{Deprez2014permanence} 
    exactness is preserved under this notion of measure equivalence. More recently, this notion was studied in more 
    detail in 
    \cite{Koivisto2017measure,KKR2018non-unimodular}. It was proved that two lcsc groups $G$ and $H$ are measure 
    equivalent if and only if they admit essentially free, ergodic pmp actions on some standard probability space for 
    which the cross section equivalence relations are stably isomorphic. Using this characterization, we were able to
    prove the following result. For countable groups it was proven by \citeauthor{Sako2009} in \cite{Sako2009}.
    \begin{theoremletter}\label{bi-exact.measure equivalent}
        The class $\mathcal S$ is closed under measure equivalence.
    \end{theoremletter}
    
    In \cite{Brothier2017rigidity}, the author proved together with Brothier and Vaes that the group von Neumann 
    algebra $L(G)$ is solid whenever $G$ is a locally compact group in class $\mathcal S$. In particular, when 
    $L(G)$ is also a 
    nonamenable factor, it follows that $L(G)$ is prime. Combining \cref{bi-exact.equivalent} with the unique prime 
    factorization results of \citeauthor{Houdayer2017unique}
    in \cite{Houdayer2017unique}, we were able to obtain the following unique prime factorization results for (tensor 
    products of) such group von Neumann algebras.
    \begin{theoremletter}\label{bi-exact.UPF.source and target}
        Let $G_1, \dots, G_n$ and $H_1, \dots, H_m$ be lcsc groups in class $\mathcal S$. Assume that all 
        $L(G_i)$ and $L(H_i)$ are nonamenable factors. Let $G = G_1\times\dots\times G_n$ and 
        $H = H_1\times \dots \times H_m$. Then,
        $L(G) = L(G_1)\otimes \dots \otimes L(G_n)$ is stably isomorphic to $L(H) = L(H_1)\otimes \dots \otimes L(H_m)$
        if and only if $m = n$ and (after relabeling) $L(G_i)$ is stably isomorphic to $L(H_i)$ for $i = 1,\dots,n$.
    \end{theoremletter}
    
    \begin{theoremletter}\label{bi-exact.UPF.source}
        Let $G_1,\dots, G_n$ be lcsc groups in class $\mathcal S$ whose group von Neumann algebras are nonamenable 
        factors. Let 
        $G = G_1\times \dots\times G_n$. Let $N_1, \dots, N_m$ be non-type I factors possessing a state with large 
        centralizer. Then, if $L(G)\cong N_1\otimes \dots \otimes N_m$, we have $m\le n$. If moreover $n=m$, then (after
        relabeling) $L(G_i)$ is stably isomorphic to $N_i$ for $i = 1,\dots,n$.
    \end{theoremletter}
    We prove these two theorems by proving that for groups $G$ in class $\mathcal S$, the group von Neumann 
    algebra $L(G)$ belongs to the class $\mathcal C_{(AO)}$ introduced in \cite{Houdayer2017unique}. 
    
    It is worthwhile to note that for many locally compact groups $G$, the group von Neumann algebra $L(G)$ is amenable 
    or even type I. For instance, the group von Neumann algebra of a connected lcsc group is always amenable by \cite[Corollary~6.9]{Connes1976}.
    However, the following group $G$ due to \citeauthor{Suzuki2016elementary} provides an example of a locally compact 
    group whose group von Neumann algebra $L(G)$ is a nonamenable type II$_\infty$ factor. 
    \begin{exampleletter}[\citeauthor{Suzuki2016elementary}]
        Let $\mathbb Z_2 = \mathbb Z / 2\mathbb Z$ act on $\mathbb F_2$ by flip of the generators.
        Then the compact group $K = \prod_{k\in \mathbb N} \mathbb Z_2$ acts on the infinite free product 
        $H = \freeprod_{k\in \mathbb N} \mathbb F_2$ by letting the $k$\textsuperscript{th} component of $K$ flip the 
        generators 
        in the $k$\textsuperscript{th} component of $H$. The semi-direct product $G = H\rtimes K$ satisfies the 
        conditions of 
        \cite[Proposition]{Suzuki2016elementary} with $K_n = \prod_{k={n+1}}^\infty \mathbb Z_2$ and 
        $L_n = (\freeprod_{k=0}^n \mathbb F_2)\rtimes K$. Hence, by 
        \cite[section on Group von Neumann algebras]{Suzuki2016elementary}, its group von Neumann algebra is a 
        nonamenable factor of type II$_\infty$. Moreover, $G$ belongs to class $\mathcal S$ since $G$ is measure 
        equivalent to $H$ and $H$ belongs to class $\mathcal S$.
    \end{exampleletter}
    
    Furthermore, certain classes of groups acting on trees have nonamenable group von Neumann algebras by 
    \cite[Theorem~C and~D]{HoudayerRaum2015asymptotic}.
    Also, \cite[Theorem~E and~F]{Raum2015c*simplicity} would provide conditions on such groups under which $L(G)$
    is a nonamenable factor. In particular, for every $q\in \mathbb Q$ with $0 < q < 1$ 
    \cite[Theorem~G]{Raum2015c*simplicity} would provide examples of groups in class $\mathcal S$ for which the group
    von Neumann algebra is a nonamenable factor of type III$_q$.
    However, due to a mistake in \cite[Lemma~5.1]{Raum2015c*simplicity}
    in that paper, there is a gap in the proofs of these results (see also 
    \cite[p~20]{Raum2019bourbaki}), and it is currently not completely clear whether these results hold as stated 
    there. 
    
    \paragraph{Acknowledgments} The author would like to thank Arnaud Brothier and Stefaan Vaes for interesting 
    discussions
    and helpful comments.
    
    \section{Preliminaries and notation}
    Throughout this article, we will assume all groups to be locally compact and second countable. We denote by 
    $\lambda_G$ the left Haar measure on such a group $G$.
    All topological spaces are assumed to be locally compact and Hausdorff. All 
    actions $G\acts X$ are assumed to be continuous.
    
    Let $X$ be a locally compact space. We denote by $M(X)$ the space of 
    complex Radon measures on $X$. We can equip this space with
    the norm of total variation, or alternatively with the weak* topology when viewing it as the dual space of $C_0(X)$.
    The Borel structure from both topologies agree.
    We denote by $M(X)^+$ the space of positive Radon measures and $\Prob(G)$ the space of Radon probability 
    measures. Suppose that a group $G$ acts on $X$, then for $g\in G$ and $\mu\in M(X)$ we denote by $g\cdot \mu$ the 
    measure defined by 
    $(g\cdot \mu)(E) = \mu(g^{-1}E)$ for all Borel sets $E\subseteq X$.
    
    \subsection{Topological amenability}
    We recall from \cite{Anantharaman-Delaroche2000} the notion of topological amenability for actions of 
    locally compact groups.
    \begin{definition}
        Let $G$ be a lcsc group, $X$ a locally compact space and $G\acts X$ a continuous action. We say that 
        $G\acts X$ is \imph[amenable!topological,topologically amenable]{(topologically) amenable} if there exists
        a net of weakly* continuous maps $\mu_i: X\to \Prob(G)$ satisfying
        \begin{align}
            \lim_{i} \norm{g\cdot \mu_i(x) - \mu_i(gx)}= 0
            \label{eq:def.amenability}
        \end{align}
        uniformly on compact sets for $x\in X$ and $g\in G$.
    \end{definition}

    By \cite[Proposition~2.2]{Anantharaman-Delaroche2000}, we have the following equivalent characterization.
    \begin{proposition}\label{amenable action.equivalent}
        Let $G$ be a lcsc group, $X$ a locally compact space and $G\acts X$ a continuous action. Then, the following 
        are equivalent
        \begin{equivalent}
            \item $G\acts X$ is amenable
            \item There exists a net $(f_i)_i$ in $C_c(X\times G)^+$ satisfying 
            \[\lim_{i}\int_G f_i(x,s)\dif s = 1\] 
            uniformly on compact sets for $x\in X$ and
            \begin{align}
                \lim_{i} \int_G |f_i(x,g^{-1}s) - f_i(gx,s)|\dif s = 0
                \label{amenable action.equivalent.to zero}
            \end{align}
            uniformly on compact sets for $x\in X$ and $g\in G$.
            \label{amenable action.equivalent.continuous compactly supported}
        \end{equivalent}
    \end{proposition}
    \begin{remark}
        Obviously, when $X$ is $\sigma$-compact, we can replace nets by sequences in the above definition and 
        proposition.
    \end{remark}
    \begin{remark}\label{amenable action.equivalent.compact}
        One can check that if $X$ is a compact space, then we can take a sequence $(f_n)_n$ in $C_c(X\times G)^+$ 
        satisfying 
        $\int_G f_n(x,s)\dif s = 1$ for every $x\in X$ and every $n\in \mathbb N$ and such that
        \eqref{amenable action.equivalent.to zero} holds. 
        
    \end{remark}
    
    The following result shows that if $X$ is a $\sigma$-compact space, then one can assume that the convergence in 
    \eqref{eq:def.amenability} is uniform
    on the whole space $X$, instead of only uniform on compact sets of $X$.
    \begin{proposition}\label{amenable action.uniform on space}
        Let $G$ be a lcsc group, $X$ a $\sigma$-compact space and $G\acts X$ a continuous action. The action $G\acts X$ 
        is amenable if and only if there exists a sequence of weakly* continuous maps $\mu_n: X\to \Prob(G)$ satisfying
        \[\lim_{n\to \infty} \norm{g\cdot \mu_n(x) - \mu_n(gx)} = 0\]
        uniformly on $x\in X$ and uniformly on compact sets for $g\in G$.
    \end{proposition}
    \begin{proof}
        Suppose that $G\acts X$ is amenable. Since $X$ is $\sigma$-compact, it suffices to construct for every compact 
        set $K\subseteq G$ and every $\eps > 0$ a weakly* continuous map $\mu: X\to \Prob(G)$ satisfying
        \begin{align}
            \norm{g\cdot \mu(x) - \mu(gx)} < \eps
            \label{eq:amenable action.uniform on space.goal}
        \end{align}
        for all $g\in K$ and all $x\in X$.
        
        So, fix a compact set $K\subseteq G$ and an $\eps > 0$. Without loss of generality, we can assume that $K$ is 
        symmetric. 
        Take an increasing sequence $(L_n)_{n\ge 1}$ of compact subsets in $X$ such that $X = \bigcup_n L_n$. Since $X$ 
        is locally compact, after 
        inductively enlarging $L_n$, we can assume that $L_n\subseteq \interior(L_{n+1})$ and 
        $gL_n\subseteq L_{n+1}$
        for every $g\in K$. Using the amenability of $G\acts X$, we can take a sequence of weakly*
        continuous maps $\nu_n: X\to \Prob(G)$ satisfying
        \[\norm{g\cdot \nu_n(x) - \nu_n(gx)} < 2^{-n}\]
        for all $g\in K$, $x\in L_n$ and $n\in \mathbb N\setminus \{0\}$. Set $L_n = \emptyset$ for $n\le 0$. 
        Fix $n\ge 1$ such that $18 / n < \eps$ and take 
        continuous functions $f_k: X\to [0,1]$ such that $f_k(x) = 1$ whenever 
        $x\in L_{k}\setminus L_{k-n}$ and $f_k(x) = 0$ whenever $x\in L_{k-n-1}$ or $x\in X\setminus L_{k+1}$.
        
        For every $x\in X$, we denote $|x| = \max\set{k\in \mathbb N}{x\notin L_k}$. Set
        \[\tilde \mu(x) = \sum_{k=0}^\infty f_k(x)\nu_k(x)
                        = f_{|x|}(x)\nu_{|x|}(x) + f_{|x|+n+1}(x) \nu_{|x|+n+1}(x) + \sum_{k=|x|+1}^{|x|+n} \nu_k(x).\]
        for $x\in X$ and define $\mu: X\to \Prob(G): x\mapsto \tilde \mu(x) / \norm{\tilde \mu(x)}$. 
        Clearly, $\mu$ is weakly* continuous. To prove that $\mu$ satisfies 
        \eqref{eq:amenable action.uniform on space.goal}, fix $x\in X$ and $g\in K$. Since 
        $gL_k\subseteq L_{k+1}$ and $g^{-1}L_k\subseteq L_{k+1}$ for every $k\in \mathbb N$, we have $|x| - 1\le |gx|\le |x|+1$ and hence
        \[\norm{g\cdot \tilde \mu(x) - \tilde \mu(gx)} 
           \le 8 + \sum_{k=|x|+1}^{|x|+n} \norm{g\cdot \nu_k(x) - \nu_k(gx)}
           \le 9,\]
        where we used that $g\in K$ and $x\in L_k$ for $k = |x|+1, \dots, |x|+n$.
        Hence,
        \[\norm{g\cdot \mu(x) - \mu(gx)}
           \le \frac{2}{\norm{\tilde \mu(x)}} \norm{g\cdot \tilde \mu(x) - \tilde \mu(gx)} 
           \le \frac{18}{n} < \eps\]
        as was required.
    \end{proof}
    
    The following result can for instance be found in \cite[Exercise~15.2.1]{Brown2008} in the case of discrete groups.
    For completeness, we include a proof for locally compact groups here.
    \begin{lemma}\label{amenable action.X cs Prob(X)}
        Let $G$ be a lcsc group, $X$ a locally compact space and $G\acts X$ a continuous action. Then, $G\acts X$ is 
        amenable if and only if the induced action $G\acts \Prob(X)$ is amenable, where $\Prob(X)$ is equipped with the
        weak* topology.
    \end{lemma}
    \begin{proof}
        Since the map $X\to \Prob(X): x\mapsto \delta_x$ is weakly* continuous and $G$-equivariant, we have that 
        amenability of $G\acts \Prob(X)$ implies amenability of $G\acts X$. 
        
        Conversely, suppose that $G\acts X$ is amenable. Let $\eta_i: X\to \Prob(G)$ be a net of maps as in the 
        definition. Then,
        \[\tilde \eta_i: \Prob(X)\to \Prob(G): \mu\mapsto \int_X \eta_i(x)\dif \mu(x)\]
        is weakly* continuous and satisfies
        \[\norm{\tilde \eta_i(g\cdot \mu) - g\cdot \tilde \eta_i(\mu)}
          \le \int_X \norm{\eta_i(gx) - g\cdot \eta_i(x)} \dif \mu(x) 
          \le \sup_{x\in X} \norm{\eta_i(gx) - g\cdot \eta_i(x)}\to 0\]
        uniformly for $\mu\in \Prob(X)$ and uniformly on compact sets for $g\in G$.
    \end{proof}
    
    \subsection{Exactness}
    The following definition of exactness was given by \citeauthor{KW1999exact} in \cite{KW1999exact}. Recall that
    a $G$-\C*-algebra is a \C*-algebra $A$ together with a $\norm{.}$-continuous action $G\acts A$ by $^*$-isomorphisms.
    We denote by $A\rtimes_r G$ its reduced crossed product.
    \begin{definition}
        A lcsc group $G$ is called \imph{exact} if for every $G$-equivariant exact sequence of $G$-C*-algebras
        \[0\to A\to B\to C\to 0,\]
        also the sequence
        \[0\to A\rtimes_r G\to B\rtimes_r G \to C\rtimes_r G\to 0\]
        is exact.
    \end{definition}
    It is an immediate consequence of this definition that the reduced group \C*-algebra $\Cred(G)$ is exact whenever
    $G$ is exact. The converse is also true for discrete groups (see \cite[Theorem~5.2]{KW1999exact}), 
    but is still open for locally compact groups. The class of exact groups is very large and contains among others all 
    (weakly) amenable groups \cite{HK1994approximation,BCL2016exactness}, linear groups \cite{GHW2005Novikov} and
    hyperbolic groups \cite{Adams1994boundary}.
    By \cite[Theorem~4.1 and Theorem~5.1]{KW1999permanence} the class of 
    exact groups is closed under taking closed subgroups and extensions. Examples of non-exact groups were given by
    \citeauthor{Gromov2003random} \cite{Gromov2003random,AD2008examples} and \citeauthor{Osajda2014small} 
    \cite{Osajda2014small}.
    
    As before, we denote by $\beta^{lu}G$ the spectrum of the algebra 
    $C_b^{lu}(G)$ of bounded left-uniformly continuous functions on $G$. The action $G\acts G$ by left-translation 
    extends uniquely to a continuous action $G\acts \beta^{lu}G$.
    By \cite[Theorem~7.2]{Anantharaman-Delaroche2000} and 
    \cite[Theorem~A]{BCL2016exactness} we have the following
    equivalent characterizations of exactness.
    \begin{theorem}\label{exact.equivalent}
        Let $G$ be a lcsc group. Then, the following are equivalent.
        \begin{equivalent}
            \item $G$ is exact,
            \item $G$ admits a continuous, amenable action on some compact 
            space,
            \item the action $G\acts \beta^{lu}G$ by left-translation is amenable.
        \end{equivalent}
    \end{theorem}

    \section{Class \texorpdfstring{$\mathcal S$}{𝓢} and boundary actions small at infinity}
    The main goal of this section is to prove \cref{bi-exact.equivalent,bi-exact.map to Prob(beta^luG)}, 
    but we first need the following equivalent 
    characterizations  of the
    second condition in the definition of class $\mathcal S$. Note that point \ref{prop(S).equivalent.S(G)} in
    the proposition 
    below is 
    property (S) in the sense of \cite{Brothier2017rigidity}.
    
    As before, we denote by $\mathcal S(G)$ the space 
    $\set{f\in L^1(G)^+}{\norm{f}_1 = 1}$ of probability measures on $G$ that are absolutely continuous with respect to 
    the Haar measure.  There 
    is an obvious $G$-equivariant norm-preserving embedding $\mathcal S(G)\hookrightarrow \Prob(G)$.
    \begin{proposition}\label{prop(S).equivalent}
        Let $G$ be a lcsc group. Then, the following are equivalent.
        \begin{equivalent}
            \item There is a $\norm{.}_1$-continuous map $\eta: G\to \mathcal S(G)$ satisfying
            \[\lim_{k\to \infty}\norm{\eta(gkh) - g\cdot \eta(k)}_1= 0\]
            uniformly on compact sets for $g,h\in G$.\label{prop(S).equivalent.S(G)}
            \item There exists a $\norm{.}$-continuous map $\eta: G\to \Prob(G)$ satisfying,
            \[\lim_{k\to \infty}\norm{\eta(gkh) - g\cdot \eta(k)} = 0\]
            uniformly on compact sets for $g,h\in G$.\label{prop(S).equivalent.Prob(G)}
            \item There exists a Borel map $\eta: G\to \Prob(G)$ satisfying 
            \[\lim_{k\to \infty}\norm{\eta(gkh) - g\cdot \eta(k)} = 0\]
            uniformly on compact sets for $g,h\in G$.\label{prop(S).equivalent.Prob(G).borel}
            \item There exists a sequence of Borel maps $\eta_n: G\to M(G)^+$ satisfying
            \[\liminf_{n\to \infty} \liminf_{k\to \infty} \norm{\eta_n(k)}  > 0\]
            and
            \[\lim_{n\to \infty} \limsup_{k\to \infty} \sup_{g,h\in K}\norm{\eta_n(gkh) - g\cdot\eta_n(k)} = 0\]
            for all compact sets $K\subseteq G$.\label{prop(S).equivalent.sequence}
        \end{equivalent}
    \end{proposition}
    \begin{proof}
        The implications 
        \ref{prop(S).equivalent.S(G)}$\Rightarrow$\ref{prop(S).equivalent.Prob(G)}$\Rightarrow$\ref{prop(S).equivalent.Prob(G).borel}$\Rightarrow$\ref{prop(S).equivalent.sequence} are trivial. We prove the reverse implications
        \ref{prop(S).equivalent.sequence}$\Rightarrow$\ref{prop(S).equivalent.Prob(G).borel}$\Rightarrow$\ref{prop(S).equivalent.Prob(G)}$\Rightarrow$\ref{prop(S).equivalent.S(G)}.
        
        First, we prove \ref{prop(S).equivalent.Prob(G)}$\Rightarrow$\ref{prop(S).equivalent.S(G)}. The proof follows
        the lines of \cite[Proposition~2.2]{Anantharaman-Delaroche2000}. Let $\eta: G\to \Prob(G)$ be as in 
        \ref{prop(S).equivalent.Prob(G)}. We construct $\tilde \eta: G\to \mathcal S(G)$ as follows. Take an 
        $f\in C_c(G)^+$ with $\int_G f(t)\dif t = 1$. Define
        \[\tilde \eta(g)(s) = \int_G f(t^{-1}s)\dif \eta(g)(t)\]
        for $s,g\in G$. One checks that $\tilde \eta(g)\in \mathcal S(G)$ for every $g\in G$ and that $\tilde \eta$ is 
        $\norm{.}_1$-continuous. For all $g,h,k\in G$ we have
        \begin{align*}
            \int_G |\tilde\eta(gkh)(s) - \tilde \eta(k)(g^{-1}s) |\dif s  
                &= \int_G \abs{\int_G f(t^{-1}s) \dif \eta(gkh)(t) - \int_G f(t^{-1}g^{-1}s) \dif \eta(k)(t)}\dif s\\
                &\le \int_G \int_G f(t^{-1}s)\dif |\eta(gkh) - g\cdot \eta(k)|(t)\dif s\\
                &= \norm{\eta(gkh) - g\cdot \eta(k)}
        \end{align*}
        which tends to zero uniformly on compact sets for $g,h\in G$ whenever $k\to \infty$. 
        
        The implication \ref{prop(S).equivalent.Prob(G).borel}$\Rightarrow$\ref{prop(S).equivalent.Prob(G)} follows 
        immediately by applying \cref{prop(S).equivalent.continuous.technical} below with $H = Y = G$ and the actions
        $G\times G\acts H$ and $G\acts Y$ defined by $(g,k)\cdot h = ghk^{-1}$ and $(g,h)\cdot y = gy$ for $g,k\in G$,
        $h\in H$ and $y\in Y$.
        
        Finally, \ref{prop(S).equivalent.sequence}$\Rightarrow$\ref{prop(S).equivalent.Prob(G).borel} follows from 
        the technical lemma \ref{prop(S).equivalent.net.technical} below applied on the spaces $X = Y = G$ with the same
        actions as above.
    \end{proof}
    
    The following is a more abstract and slightly more general version of the trick in 
    \cite[Exercise~15.1.1]{Brown2008}. It will be used several
    times in this article.
    \begin{lemma}\label{prop(S).equivalent.net.technical}
        Let $X$ and $Y$ be $\sigma$-compact spaces and $G$ a lcsc group. Suppose that $G\acts X$ and $G\acts Y$ are 
        continuous 
        actions. If there exists a sequence of Borel maps $\eta_n: X\to M(Y)^+$ satisfying
        \[\lim_{n\to \infty} \limsup_{x\to \infty} \sup_{g\in K}\norm{\eta_n(gx) - g\cdot \eta_n(x)} = 0\]
        for all compact sets $K\subseteq G$ and
        \[\liminf_{n\to \infty} \liminf_{x\to \infty} \norm{\eta_n(x)} > 0.\]
        Then, there exists a Borel map $\eta: X\to \Prob(Y)$ such that
        \begin{align}
        \lim_{x\to \infty} \norm{\eta(gx) - g\cdot \eta(x)} = 0
        \label{eq:prop(S).equivalent.net.technical.result}
        \end{align}
        uniformly on compact sets for $g\in G$. Moreover, if the maps $\eta_n$ are assumed to be $\norm{.}$-continuous 
        (resp. weakly* continuous), then also $\eta$ can be assumed to be $\norm{.}$-continuous 
        (resp. weakly* continuous).
    \end{lemma}
    \begin{proof}
        After passing to a subsequence and replacing values of $\eta_n$ in a compact set, we can assume that there 
        exists a 
        $\delta > 0$ such that $\norm{\eta_n(x)}_1 \ge \delta$ for all $n\in \mathbb N$ and all $x\in X$. Set 
        $\tilde \eta_n(x) = \eta_n(x) / \norm{\eta_n(x)}$ for all $x\in X$. Note that
        \[\lim_{n\to \infty} \limsup_{x\to \infty} \sup_{g\in K}\norm{\tilde \eta_n(gx) - g\cdot \tilde \eta_n(x)} 
        \le \lim_{n\to \infty} \limsup_{x\to \infty} \sup_{g\in K}\frac{2}{\norm{\eta_n(x)}} \norm{\eta_n(gx) - g\cdot \eta_n(x)} = 0\]
        for all compact sets $K\subseteq G$.
        
        Take an increasing sequence $(K_n)_n$ of compact symmetric neighborhoods of the unit $e$ in $G$
        such that $G = \bigcup_n \inter(K_n)$.
        After passing to a subsequence of $(\tilde \eta_n)_n$, we find compact sets $L_n\subseteq X$ 
        such that
        \[\norm{\tilde\eta_n(gx) - g\cdot \tilde \eta_n(x)}\le 2^{-n+1}\]
        for all $g\in K_n$  and $x\in X\setminus L_n$. After inductively enlarging $L_n$, we can assume 
        that the sequence $(L_n)_n$ is increasing, that
        $gL_n\subseteq L_{n+1}$ for all $g\in K_n$ and 
        that $X = \bigcup_n L_n$. Moreover, we can also assume $L_0$ to be the empty set.
        
        For every $x\in X$, we denote 
        $|x| = \max\set{n\in \mathbb N}{x\notin L_n}$. Furthermore, we denote $h(n) = \floor{n/2} + 1$ for $n\ge 1$. 
        Fix a $y_0\in Y$. We set $\mu(x) = \delta_{y_0}$ whenever $|x|\le 1$ and
        \[\mu(x) = \sum_{k=h(|x|)}^{|x|} \tilde \eta_k(x)\]
        whenever $|x| \ge 2$.
        Now, define $\eta: X\to \Prob(Y)$ by $\eta(x) = \mu(x) / \norm{\mu(x)}$.
        
        To prove that $\eta$ satisfies \eqref{eq:prop(S).equivalent.net.technical.result}, take $\eps > 0$ and 
        $K\subseteq G$ arbitrary.
        Since $\bigcup_n \inter(K_n) = G$, we can take an $n_0 \ge 1$
        such that $K\subseteq K_{n_0}$. Take $n_1 > \max\{2n_0,16 / \eps\}$. 
        We claim that 
        $\norm{\eta(g\cdot x) - g\cdot \eta(x)} < \eps$ whenever $x\in G\setminus L_{n_1}$ and $g\in K$. Indeed, fix 
        $g\in K$ and $x\in G\setminus L_{n_1}$. Take $n \ge n_1$ such that $x\in L_{n+1}\setminus L_{n}$. Then, 
        $|x| = n$. Since 
        $gL_{n+1}\subseteq L_{n+2}$ and $g^{-1}L_{n-1}\subseteq L_{n}$, we have that 
        $gx\in L_{n+2}\setminus L_{n-1}$ and 
        hence $n-1\le |gx| \le n+1$. This yields
        \[\norm{\mu(gx) - g\cdot \mu(x)}
        \le 2 + \sum_{k=h(n)}^{n} \norm{\tilde \eta_k(gx) - g\cdot \tilde \eta_k(x)}
        \le 4,\]
        since $g\in K_k$ and $x\in X\setminus L_k$ for $k = h(n),\dots, n$.
        Hence
        \[\norm{\eta(gx) - g\cdot \eta(x)}
        \le \frac{2}{\norm{\mu(x)}} \norm{ \mu(gx) - g\cdot \mu(x) }
        \le \frac{4}{n}\cdot 4 < \eps\]
        which proves the claim.
        
        If the maps $\eta_n$ are $\norm{.}$-continuous, we can make $\mu$ (and hence $\eta$)
        $\norm{.}$-continuous in the following way. By inductively enlarging the compact 
        sets $L_n$ above and using that $X$ is locally compact, we can assume that $L_n\subseteq \interior(L_{n+1})$. 
        For all $n\ge 1$, we 
        take a continuous function
        $f_n: X\to [0,1]$ such that $f_n(x) = 1$ if $x\in L_{2n}\setminus L_{n}$ and $f_n(x) = 0$ if 
        $x\in L_{n-1}$ or $x\in X\setminus L_{2n+1}$.        
        For $x\in X$ with $|x|\ge 2$, we set 
        \[\tilde \mu(x) = \sum_{k=1}^{+\infty} f_k(x) \tilde \eta_k(x) 
                 = f_{h(|x|)-1}(x)\;\tilde \eta_{h(|x|)-1}(x) + 
                   f_{|x|+1}(x)\;\tilde \eta_{|x|+1}(x) + \sum_{k=h(|x|)}^{|x|} \tilde \eta_k(x).\]
        Fix again $y_0\in Y$ and take a continuous map $a: X\to [0,1]$ such that $a(x) = 1$ if $x\in L_2$ and 
        $a(x) = 0$ when $x\in X\setminus L_3$. We define the continuous map 
        $\eta: X\to M(Y)^+$ by
        \[\mu(x) = \begin{cases}
                        a(x)\delta_{y_0} + \big(1-a(x)\big)\frac{\tilde \mu(x)}{\norm{\tilde \mu(x)}}&\text{if $|x|\ge 2$}\\
                        \delta_{y_0}&\text{if $|x|\le 1$}
                   \end{cases}\]
        Obviously, $\eta$ is 
        $\norm{.}$-continuous and, by a similar calculation as above, one proves that $\eta$ 
        satisfies \eqref{eq:prop(S).equivalent.net.technical.result}.
    \end{proof}
    \begin{remark}\label{prop(S).equivalent.technical.more general}
        Using almost exactly the same proof as above, one can actually prove the following slightly more general 
        result: suppose that for every $\eps > 0$, every compact set $K\subseteq G$, there exists a compact set 
        $L\subseteq X$ such that for all compact sets $L'\subseteq X$, there exists a map $\mu: X\to M(Y)^+$ such
        that
        \begin{align}
            \frac{\norm{\mu(gx) - g\cdot \mu(x)}}{\norm{\mu(x)}} < \eps
            \label{eq:prop(S).equivalent.technical.more general}
        \end{align}
        whenever $g\in K$ and $x\in L'\setminus L$. Then, there exists a map $\eta: X \to \Prob(Y)$ as in 
        \eqref{eq:prop(S).equivalent.net.technical.result}. Indeed, using the notation of the proof, we can take the compact 
        sets $L_n\subseteq X$ and the maps
        $\eta_n: X\to M(Y)^+$ such that
        \[\norm{\tilde \eta_n(gx) - g\cdot \tilde \eta_n(x)}\le \frac{2}{\norm{\eta_n(x)}} \norm{\eta_n(gx) - g\cdot \eta_n(x)} < 2^{-n+1}\]
        for all $g\in K_n$ and $x\in L_{2n}\setminus L_n$, where again 
        $\tilde \eta_n(x) = \eta_n(x) / \norm{\eta_n(x)}$. The rest of the proof holds verbatim.
    \end{remark}
    
    \pagebreak[3]
    The following lemma will be used several times to replace Borel maps by continuous maps.
    \begin{lemma}\label{prop(S).equivalent.continuous.technical}
        Let $H$ and $G$ be lcsc groups and $Y$ a locally compact space. Suppose that $G\acts^\alpha H$ is a continuous 
        action by group 
        automorphisms and that $G\acts Y$ is continuous. If there exists a Borel map $\eta: H\to \Prob(Y)$
        satisfying
        \[\lim_{h\to \infty} \norm{\eta\big(\alpha_g(h)\big) - g\cdot \eta(h)} = 0
          \mathand \lim_{h\to \infty} \norm{\eta\big(\alpha_g(h)k\big) - \eta\big(\alpha_g(hk)\big)} = 0\]
        uniformly on compact sets for $g\in G$ and $k\in H$, then there exists a $\norm{.}$-continuous map 
        $\tilde \eta: H\to \Prob(Y)$ map satisfying
        \[\lim_{h\to \infty} \norm{\tilde \eta\big(\alpha_g(h)\big) - g\cdot \tilde \eta(h)} = 0\]
    \end{lemma}
    \begin{proof}
        Fix a compact neighborhood $K$ of the unit $e$ in $H$ with $\lambda_H(K) = 1$. We define 
        $\tilde \eta: H\to \Prob(Y)$ by
        \[\tilde \eta(g) = \int_K \eta(gk)\dif k.\]
        The map $\tilde \eta$ is continuous, since for $h_1,h_2\in H$ we have
        \[\norm{\tilde{\eta}(h_1) - \tilde{\eta}(h_2)} \le \int_{h_1K\triangle h_2K} \norm{\eta(k)}\dif k 
        = \lambda_H(h_1K\triangle h_2K)\]
        and the right hand side tends to zero whenever $h_2\to h_1$. Moreover, for $g\in G$ and $h\in H$, we have
        \[\norm{\tilde \eta \big(\alpha_g(h)\big) - g\cdot \tilde \eta(h)}
            \le \int_K \norm{  \eta\big(\alpha_g(h)k\big) -  g\cdot \eta(hk) }\dif k\]
        Since $K$ is compact the right hand side tends to zero uniformly on compact sets 
        for $g\in G$ whenever $h \to \infty$.
    \end{proof}
    
    We are now ready to prove \cref{bi-exact.equivalent}.
    \begin{proof}[Proof of \cref{bi-exact.equivalent}]
        First, we prove \ref{bi-exact.equivalent.bi-exact}$\Rightarrow$\ref{bi-exact.equivalent.boundary}. Let 
        $\eta: G\to \Prob(G)$ be a map as in the definition of class $\mathcal S$. Consider
        the u.c.p. map $\eta_*: C_b^{lu}(G)\to C_b^u(G)$ defined by
        \[(\eta_*f)(g) = \int_G f(s)\dif \eta(g)(s)\]
        for $f\in C_b^{lu}(G)$ and $g\in G$. Note that $\eta_*$ is well-defined. Indeed, fix $f\in C_b^{lu}(G)$ and 
        $\eps > 0$. 
        For $g,h\in G$ we have
        \[|(\eta_*f)(h^{-1}g) - (\eta_*f)(g)|\le \min\left\{ \norm{f}_\infty\norm{\eta(h^{-1}g) - \eta(g)},\  \norm{f}_\infty \norm{\eta(h^{-1}g) - h^{-1}\cdot \eta(g)} + \norm{f - \lambda_hf}_\infty \right\}\]
        and
        \begin{align}
            |(\eta_*f)(gh) - (\eta_*f)(g)| \le \norm{f}_\infty\norm{\eta(gh) - \eta(g)}.\label{eq:bi-exact.equivalent.right uniformly continuous}
        \end{align}
        Pick a compact neighborhood $K$ of the unit $e$ in $G$. We find a compact subset $L\subseteq G$ such that
        \[\norm{\eta(gh) - \eta(g)} \le \eps\mathand 
          \norm{\eta(h^{-1}g) - h^{-1}\cdot \eta(g)}\le \frac \eps2\]
        whenever $h\in K$ and $g\in G\setminus L$. Now, we can take an open neighborhood $\mathcal U\subseteq K$ of the 
        unit $e$ in $G$
        such that
        \[\norm{f - \lambda_hf}_\infty \le \frac \eps2,\qquad
          \sup_{g\in L}\norm{\eta(h^{-1}g) - \eta(g)}\le \eps\mathand
          \sup_{g\in L}\norm{\eta(gh) - \eta(h)}\le \eps\]
        for all $h\in \mathcal U$. It follows that
        \[\norm{\lambda_h(\eta_*f)-\eta_*f}_\infty\le \eps\mathand\norm{\rho_h(\eta_*f) - \eta_*f}_\infty < \eps\]
        for all $h\in \mathcal U$.
        
        Moreover, \eqref{eq:bi-exact.equivalent.right uniformly continuous} also implies that
        \[\lim_{g\to \infty} |(\eta_*f)(gh) - (\eta_*f)(g)| = 0\]
        for all $h\in G$ and hence that $\eta_*(f)\in C(h^u G)$ for all $f\in C_b^{lu}(G)$. 
        
        Similarly, one proves that $\eta_*(\lambda_gf) - \lambda_g(\eta_*f)\in C_0(G)$. Let 
        $\pi: C(h^u G) \to C(\nu^u G)\cong C(h^u G) / C_0(G)$ be the quotient map. It follows that
        $\pi\circ \eta_*: C_b^{lu}(G)\to C(\nu^u G)$ is a $G$-equivariant u.c.p. map. Dualizing, this map induces a 
        weakly* 
        continuous $G$-equivariant map
        $\Prob(\nu^u G)\to \Prob(\beta^{lu}G)$ given by $\mu\mapsto \mu\circ \pi\circ \eta_*$. Since $G$ is exact,
        the action $G\acts \beta^{lu}G$ is amenable and hence so is $G\acts \Prob(\beta^{lu}G)$ (see 
        \cref{amenable action.X cs Prob(X)}). It follows that $G\acts \Prob(\nu^uG)$ is amenable and hence 
        so is $G\acts \nu^uG$.
        
        Now, we prove \ref{bi-exact.equivalent.boundary}$\Leftrightarrow$\ref{bi-exact.equivalent.compactification}. 
        The implication from right to left is trivial. To prove the other implication, take an arbitrary compact 
        subset $K\subseteq G$  and an $\eps > 0$. By
        \cref{amenable action.equivalent}, it suffices to 
        construct a function $h\in C_c(h^uG\times G)^+$ such that $\int_G h(x,s)\dif s = 1$ for every 
        $x\in h^uG$ and
        \begin{align}
            \int_G | h(x,g^{-1}s) - h(gx,s)|\dif s < \eps
            \label{eq:bi-exact.equivalent.h^uG.goal}
        \end{align}
        for all $x\in h^uG$ and $g\in K$.
        
        By \cref{amenable action.equivalent,amenable action.equivalent.compact}, we find an 
        $f\in C_c(\nu^uG\times G)^+$ satisfying $\int_G f(x,s)\dif s = 1$ and
        \[\int_G |f(x,g^{-1}s) - f(gx,s)|\dif s  <\frac \eps 2 \]
        for all $x\in \nu^uG$ and $g\in K$. By the Tietze Extension Theorem, we can extend $f$ to a function 
        $\tilde f\in C_c(h^uG\times G)^+$. Since $\tilde f = f$ on $\nu^uG\times G$, we can take a compact set 
        $L\subseteq G$ and renormalize $\tilde f$ such that
        \[\int_G \tilde f(x,s)\dif s = 1\mathand \int_G | \tilde f(x,g^{-1}s) - \tilde f(gx,s)|\dif s < \frac \eps 2 \]
        for all $x\in h^uG\setminus L$ and $g\in K$. 
        
        Now, fix a function $a\in C_c(G)^+$ with $\int_G a(s)\dif s = 1$. Using \cref{slowly changing function} below,
        we take a function $\zeta\in C_c(G)^+$ such that $\zeta|_L = 1$ and $|\zeta(gh) - \zeta(h)| < \eps / 4$ for 
        $h\in G$ and $g\in K$. Now, define $h\in C_c(h^uG\times G)$ by
        \[h(x,s) = \begin{cases}
                        \zeta(x) a(x^{-1}s) + \big(1-\zeta(x)\big) \tilde f(x,s)&\text{if $x\in G$,}\\
                        \tilde f(x,s)&\text{if $x\in \nu^uG$.}
                   \end{cases}\]
        A straightforward calculation shows that $h$ satisfies \eqref{eq:bi-exact.equivalent.h^uG.goal}.
        
        Next, we prove \ref{bi-exact.equivalent.boundary}$\Rightarrow$\ref{bi-exact.equivalent.GxG action}
        Denote by $X$ the spectrum of $A = C_b^u(G) / C_0(G)$. Since $C(h^uG)\subseteq C_b^u(G)$, we have a natural 
        embedding
        $C(\nu^uG)\hookrightarrow A$, which in turn induces a continuous map $\varphi_\ell: X\to \nu^u G$. Note that 
        $\varphi_\ell$ is $G\times G$-equivariant with respect to the actions induced by left and right translation.
        Similarly, we get a $G\times G$-equivariant map $\varphi_r: X\to \nu^{u}_rG$, where $\nu_r^uG$ denotes the 
        spectrum of the algebra
        \[C(\nu_r^uG) = \set{f\in C^u(G)}{\lambda_gf - f\in C_0(G)}\]
        and the action $G\times G\acts \nu_r^uG$ is induced by left and right translation. 
        By assumption, the action $G\times 1\acts \nu^u G$ is amenable, and by symmetry so is $1\times G\acts \nu_r^uG$.
        Since the actions $1\times G\acts \nu^uG$ and $G\times 1 \acts \nu_r^uG$ are trivial, the diagonal action 
        $G\times G\acts \nu^u G\times \nu_r^uG$ is amenable. Now, the conclusion follows from the 
        $G\times G$-equivariance of the map 
        $\varphi_\ell\times \varphi_r: X\to \nu^u G\times \nu_r^uG$.
        
        Finally, we prove \ref{bi-exact.equivalent.GxG action}$\Rightarrow$\ref{bi-exact.equivalent.bi-exact}. By 
        \cref{exact.equivalent}, the group $G$ is exact. Denote 
        again by $X$ the spectrum of $A = C_b^u(G) / C_0(G)$. Denoting by $\beta^uG$ the spectrum of $C_b^u(G)$, we get 
        $X = \beta^uG\setminus G$. By 
        \cref{amenable action.equivalent,amenable action.equivalent.compact}, we can take 
        a sequence $(f_n)_{n}$ of functions in 
        $C_c(X\times G\times G)^+$ 
        such that
        $\int_{G\times G} f_n(x,s,t)\dif s\dif t = 1$ for all $x\in X$ and $n\in \mathbb N$, and such that
        \begin{align}
            \lim_{n\to \infty}\int_{G\times G} |f_n(x,g^{-1}s,h^{-1}t) - f_n\big((g,h)\cdot x,s,t\big)|\dif s\dif t = 0
            \label{eq:bi-exact.equivalent.GxG action.f_i}
        \end{align}
        uniformly for $x\in X$ and uniformly on compact sets for $g,h\in G$. As before, the Tietze Extension Theorem
        yields extensions $\tilde f_n\in C_c(\beta^uG\times G\times G)^+$ of each $f_n$. For each $x\in \beta^uG$ and 
        $n\in \mathbb N$, we define $\eta_n(x)\in M(G)^+$ as the measure with density function 
        \[s\mapsto \int_G \tilde f_n(x,s,t)\dif t.\]
        with respect to the Haar measure.
        This yields $\norm{.}$-continuous maps $\eta_n: \beta^uG\to M(G)^+$. By 
        \eqref{eq:bi-exact.equivalent.GxG action.f_i}, the restrictions of $\eta_n$ to $G\subseteq \beta^uG$ satisfy 
        the conditions of \cref{prop(S).equivalent} \ref{prop(S).equivalent.sequence}.
    \end{proof}
    
    In the proof above, we used the following easy lemma.
    \begin{lemma}\label{slowly changing function}
        Let $G$ be a lcsc group. For all compact subsets $K,L\subseteq G$ and all $\eps > 0$, there exists a continuous 
        function 
        $f\in C_c(G)$ satisfying $f|_L = 1$ and
        \[|f(kgk') - f(g)| < \eps\]
        for $k,k'\in K$ and $g\in G$.
    \end{lemma}
    \begin{proof}
        Without loss of generality we can assume that $K$ is symmetric and that $\inter(K)$ contains the unit $e$. 
        Denote $L_0 = L$ and 
        $L_n = K^nLK^n$ for $n\ge 1$. Then, 
        $L_n\subseteq \inter(L_{n+1})$ for every $n\in \mathbb N$ and hence, we can take a continuous $f_n:G\to [0,1]$ 
        with $f_n(g) = 1$ 
        for $g\in L_n$ and $\supp f_n\subseteq L_{n+1}$. Take $N\in \mathbb N$ such that $1/N< \eps / 4$ and set
        \[f(g) = \frac 1N\sum_{k=0}^{N-1}f_k(g).\]
        Then, $f$ satisfies the conclusions of the lemma.
    \end{proof}
    
    We end this section by proving \cref{bi-exact.map to Prob(beta^luG)}.
    \begin{proof}[Proof of \cref{bi-exact.map to Prob(beta^luG)}]
        The implication from left to right is trivial. To prove the converse implication, note that by exactness of $G$
        and \cref{amenable action.X cs Prob(X)}, the action $G\acts\Prob(\beta^{lu}G)$ is amenable. Take a sequence
        $\theta_n: \Prob(\beta^{lu}G)\to \Prob(G)$ such that
        \[\lim_{n\to \infty} \norm{\theta_n(g\cdot \mu) - g\cdot \theta_n(\mu) } = 0\]
        uniformly for $\mu \in \Prob(\beta^{lu}G)$ and uniformly on compact sets for $g\in G$. Now, for the composition
        $\eta_n = \theta_n\circ \eta$ we get
        \begin{align*}
            \norm{\eta_n(gkh) - g\cdot \eta_n(k)}
              &\le \norm{\eta(gkh) - g\cdot \eta(k)} + \norm{ \theta_n\big(g\cdot \eta(k) \big) - g\cdot \theta_n\big( \eta(k) \big) }\\
              &\le \norm{\eta(gkh) - g\cdot \eta(k)} + \sup_{\mu\in \Prob(\beta^{lu}G)} \norm{ \theta_n(g\cdot \mu) - g\cdot \theta_n\big( \mu \big) }
        \end{align*}
        whenever $g,h,k\in G$. It follows that
        \[\lim_{n\to \infty}\limsup_{k\to \infty} \sup_{g,h\in K}\norm{ \eta_n(gkh) - g\cdot \eta_n(k) } = 0\]
        for every compact set $K\subseteq G$.
        Hence, \cref{prop(S).equivalent.net.technical} concludes the proof.
    \end{proof}
    
    \section{Examples of groups in class \texorpdfstring{$\mathcal S$}{𝓢}}
    In this section, we prove \cref{bi-exact.R^2 rtimes SL_2(R),bi-exact.wreath product.introduction}. Before we
    can start the proof of these results, we need a few lemmas. The first lemma is a locally compact version
    of \cite[Lemma~15.2.6]{Brown2008}. This result can be proven in a similar way as in 
    \cite[Lemma~15.2.6]{Brown2008}. However, we provide a different proof not requiring exactness.
    
    \begin{proposition}\label{bi-exact.amenable subgroup}
        Let $G$ be an lcsc group and $K$ a closed, amenable subgroup. If there exists a Borel map 
        $\eta: G\to \Prob(G / H)$ 
        such that
        \[\lim_{k\to \infty} \norm{\eta(gkh) - g\cdot \eta(k)} = 0\]
        uniformly on compact sets for $g,h\in G$. Then, $G$ has property (S), i.e. there exists a 
        $\norm{.}$-continuous map 
        $\tilde \eta: G\to \Prob(G)$ satisfying \eqref{eq:prop (S)}.
    \end{proposition}
    \begin{proof}
        Using \cref{prop(S).equivalent.continuous.technical} we can assume that $\eta$ is $\norm{.}$-continuous. 
        The proof then follows easily from \cref{bi-exact.amenable subgroup.technical} below.
    \end{proof}
    
    Let $G$ be a group and $H\subseteq G$ a closed subgroup. Denote by $p: G \to G/ H$ the  quotient map. Let
    $\sigma: G / H \to G$ be a locally bounded Borel cross 
    section for $p$, i.e. a Borel map satisfying $p \circ \sigma = \Id_{G/H}$ that maps compact sets onto 
    precompact sets (see for instance \cite[Lemma~1.1]{Mackey1952representations} for the existence of such a map). 
    We can identify $G$ with $G / H \times H$ via the map the map 
    \begin{align}
        \phi: G\to G / H\times H: g\mapsto \big(gH,\sigma(gH)^{-1}g\big).
        \label{eq:identification.G and G/H times H}
    \end{align}
    Under this 
    identification the action by left translation is given by 
    $k\cdot (gH, h) = \big(kgH, \omega(k,gH) h\big)$, where $\omega(k,gH) = \sigma(kgH)^{-1}k\sigma(gH)$.
    Note that $\omega$ maps compact sets of $G \times G / H$ onto precompact sets of $G$. 
    
    The identification map $\phi$ is not continuous, but it is bi-measurable and maps (pre)compact sets to precompact 
    sets.
    This allows us to identify the spaces $\Prob(G)$ and $\Prob(G/H\times H)$ via the map $\mu \mapsto \phi_*\mu$.
    Note that this identification map is continuous with respect to the norm topology on both spaces (and hence 
    bi-measurable), but not with respect to the weak* topology on both spaces.
    We use the above identifications in the following two lemmas.
    \begin{lemma}\label{bi-exact.amenable subgroup.sequence}
        Let $G$ be a lcsc group and $H\subseteq G$ a closed, amenable subgroup. Let $(\nu_n)_n$ be a sequence in 
        $\Prob(H)$ satisfying 
        \[\lim_{n\to \infty}\norm{h\cdot \nu_n - \nu_n} = 0\]
        uniformly on compact sets for $h\in H$. Then,
        \[\lim_{n\to \infty} \norm{h\cdot (\mu \otimes \nu_n) - (h\cdot \mu) \otimes \nu_n} = 0\]
        uniformly on compact sets for $g\in G$ and $\mu\in \Prob(G/H)$, where we equipped $\Prob(G/H)$ with the weak*
        topology.
    \end{lemma}
    \begin{proof}
        Fix compact subsets $K\subseteq G$ and 
        $\mathcal L\subseteq \Prob(G/H)$. Take an arbitrary $\eps > 0$. There is a compact subset $L\subseteq G / H$
        such that $\mu(L)> 1-\eps$ for all $\mu\in \mathcal L$. Hence, for all $f\in C_c(G/H\times H)$, 
        $\mu\in \mathcal L$, $k\in G$ and $n\in \mathbb N$, we have
        \begin{align*}
        &\abs{\int_{G/H \times H}  f\dif k\cdot (\mu\otimes \nu_n) - \int_{G/H\times H} f\dif\, (k \cdot\mu)\otimes \nu_n}\\
        &\qquad= \abs{\int_{G/H} \int_H f\big(k\cdot (gH,h)\big) \dif \nu_n(h) \dif \mu(gH) - \int_{G/H} \int_H f(kgH,h) \dif \nu_n(h)\dif \mu(gH)}\\
        &\qquad= \abs{\int_{G/H} \int_H f\big(kgH, \omega(k,gH) h\big) \dif \nu_n(h) \dif \mu(gH) - \int_{G/H} \int_H f(kgH,h) \dif \nu_n(h)\dif \mu(gH)}\\
        &\qquad\le \int_{G/H} \int_H |f(kgH, h)| \dif \big| \omega(k,gH)\cdot \nu_n - \nu_n\big|(h)\dif \mu(gH)\\
        &\qquad\le \norm{f}_\infty \left(2 \eps + \int_L \norm{\omega(k,gH)\cdot \nu_n - \nu_n}\dif \mu(gH) \right),
        \end{align*}
        where $\big| \omega(k,gH)\cdot \nu_n - \nu_n\big|$ denotes the total variation measure of 
        $\omega(k,gH)\cdot \nu_n - \nu_n$. Since $\omega$ maps compact sets to precompact sets, we can find an 
        $n_0\in \mathbb N$ such that $\norm{\omega(k,gH)\cdot \nu_n - \nu_n} < \eps$ for all $n\ge n_0$, 
        all $k\in K$ and all $gH\in L$. We conclude that
        \[\norm{k\cdot (\mu\otimes \nu_n) - (k\cdot \mu)\otimes \nu_n} \le 3\eps\]
        whenever  $n\ge n_0$, $\mu\in \mathcal L$ and $k\in K$, thus proving the result.
    \end{proof}
    
    \begin{lemma}\label{bi-exact.amenable subgroup.technical}
        Let $G$ and $H$ be lcsc groups, $\pi:G\to H$ a continuous morphism and $K\subseteq H$ a closed, amenable 
        subgroup. Let $G\acts X$ be a continuous action on some $\sigma$-compact space $X$. Let $G\acts \Prob(H)$ 
        (resp. 
        $G\acts \Prob(H / K)\:$) be defined by $g\cdot \mu = \pi(g)\cdot \mu$ for $g\in G$ and $\mu\in \Prob(H)$ 
        (resp. $\mu\in \Prob(H/K)\:$).
        If there exists a weakly* continuous map $\eta: X\to \Prob(H/K)$ such that
        \[\lim_{x\to \infty} \norm{\eta(gx) - g\cdot \eta(x)} = 0\]
        uniformly on compact sets for $g\in G$. Then, there exists a Borel map 
        $\tilde \eta: X\to \Prob(H)$ such 
        that
        \[\lim_{x\to \infty} \norm{\tilde \eta(gx) - g\cdot \tilde \eta(x)} = 0\]
        uniformly on compact sets for $g\in G$. Moreover, if $\eta$ is assumed to be $\norm{.}$-continuous then also 
        $\tilde \eta$ can be assumed to be $\norm{.}$-continuous.
    \end{lemma}
    \begin{proof}
        Fixing a locally bounded Borel cross section $\sigma: H/K\to H$ for the quotient map $p: H\to H / K$,
        we can identify $H$
        with $H / K\times K$ and $\Prob(H)$ with $\Prob(H/K\times K)$ as in \eqref{eq:identification.G and G/H times H}.
        
        Since $K$ is amenable, we can take a sequence $(\nu_n)_n$ in $\Prob(K)$ such that 
        $\norm{k\cdot \nu_n - \nu_n}\to 0$ uniformly on compact sets for $k\in K$ whenever $n\to \infty$. Using 
        \cref{bi-exact.amenable subgroup.sequence}, we construct
        maps as in \cref{prop(S).equivalent.technical.more general}
        as follows. Fix an $\eps > 0$ and a compact $C\subseteq G$. Take a compact $L\subseteq X$ such that 
        $\norm{\eta(gx) - g\cdot \eta(x)} < \eps$ for all $g\in C$ and $x\in X\setminus L$. Fix any compact 
        set $L'\subseteq X$. Applying \cref{bi-exact.amenable subgroup.sequence} to the weak* compact set $\eta(L')$, 
        we find an $n\in \mathbb N$ such 
        that
        \[\norm{\big(g\cdot \eta(x)\big) \otimes \nu_n - g\cdot \big(\eta(x)\otimes \nu_n\big) } < \eps\]
        for any $x\in L'$ and $g\in C$. Hence,
        \[\norm{\eta(gx) \otimes \nu_n - g\cdot \big(\eta(x)\otimes \nu_n\big) } 
          \le \norm{\eta(gx) - g\cdot \eta(x)} + 
              \norm{\big(g\cdot \eta(x)\big) \otimes \nu_n - g\cdot \big(\eta(x)\otimes \nu_n\big) } \le 2\eps\]
        for any $g\in C$ and any $x\in L'\setminus L$. We conclude that the map $\mu: X\to \Prob(H)$ defined by  
        $\mu(x) = \eta(x)\otimes \nu_n$ is as in \eqref{eq:prop(S).equivalent.technical.more general}. 
        Moreover, if $\eta$ is 
        $\norm{.}$-continuous, then so is $\mu$.
    \end{proof}
    
    The second result that we need before proving 
    \cref{bi-exact.R^2 rtimes SL_2(R),bi-exact.wreath product.introduction} characterizes when a semi-direct
    product belongs to class $\mathcal S$. By definition a semi-direct product $G = B\rtimes H$ belongs to class 
    $\mathcal S$ whenever it is exact and there exists a map $\eta: G\to \Prob(G)$ satisfying 
    \begin{align*}
        \norm{\mu\big((a,k)(b,h)(a',k')\big) - (a,k)\cdot \mu(b,h)} \to 0
    \end{align*}
    uniformly on compact sets for $(a,k), (a',k')\in G$ whenever $(b,h)\to \infty$. The result below shows that is 
    suffices that there exist \emph{two} such maps one of which satisfies the convergence above when $b\to \infty$ and
    the other when $h\to \infty$.
    \begin{proposition}\label{bi-exact.semi-direct product.more general}
        Let $G = B\rtimes_\alpha H$ be a semi-direct product of lcsc groups. Then, $G$ is in class $\mathcal S$ if and
        only if $B$ and $H$ are exact, and there exists Borel maps $\mu: G\to \Prob(G)$ and $\nu: G\to \Prob(G)$ 
        such that
        \begin{align}
            \lim_{b\to \infty}\norm{\mu\big((a,k)(b,h)(a',k')\big) - (a,k)\cdot \mu(b,h)} = 0
            \label{bi-exact.semi-direct product.more general.b to infty}
        \end{align}
        uniformly on compact sets for $a,a'\in B$ and $k,h,k'\in H$, and such that
        \begin{align}
            \lim_{h\to \infty} \norm{\nu\big((a,k)(b,h)(a',k')\big) - (a,k)\cdot \nu(b,k)}
            \label{bi-exact.semi-direct product.more general.h to infty}
        \end{align}
        uniformly for $b\in B$ and uniformly on compact sets for $a,a'\in B$ and $k,k'\in H$.
    \end{proposition}
    \begin{proof}
        The only if part is immediate. Indeed, the map $\eta: G\to \Prob(G)$ as in the definition of class 
        $\mathcal S$, satisfies both \eqref{bi-exact.semi-direct product.more general.b to infty} and 
        \eqref{bi-exact.semi-direct product.more general.h to infty}. Moreover, the groups $B$ and $H$ are exact as 
        subgroups of an exact group.
        
        To prove the converse, note first that $G$ is exact as an extension of an exact group by an exact group (see 
        \cite[Theorem~5.1]{KW1999permanence}). Let 
        $\mu: G\to \Prob(G)$ and $\nu: G\to \Prob(G)$ be as above. By \cref{prop(S).equivalent}, it suffices to prove
        that for every compact $K\subseteq G$ and every $\eps > 0$, there exists a Borel map $\eta: G\to \Prob(G)$
        and a compact $L\subseteq G$ such that
        \begin{align}
            \norm{\eta\big((a,k)(b,h)(a',k')\big) - (a,k)\cdot \eta(b,h)} < \eps
            \label{eq.bi-exact.semi-direct prod.TP}
        \end{align}
        for all $(a,k), (a',k')\in K$ and all $(b,h)\in G\setminus L$.
        
        So, fix a compact $K\subseteq G$ and an $\eps > 0$. 
        Let $K_B\subseteq B$ and $K_H\subseteq H$ be compact subsets such that 
        $K\subseteq \set{(b,h)}{b\in K_B, h\in K_H}$.
        By assumption, we can take a compact set $\widetilde L_H\subseteq H$ such that
        \begin{align}
            \norm{\nu\big((a,k)(b,h)(a',k')\big) - (a,k)\cdot \nu(b,k)} < \frac \eps 2
            \label{eq:bi-exact.semi-direct product.nu}
        \end{align}
        whenever $a,a'\in K_B$, $b\in B$, $k,k'\in K_H$ and $h\in H\setminus \widetilde L_H$. 
        
        Using \cref{slowly changing function}, we take a function $f\in C_c(H)$ such that $f(h) = 1$ for 
        $h\in \widetilde L_H$ and $|f(khk') - f(h)| < \eps / 4$ whenever $h\in H$ and $k,k'\in K_H$. Set 
        $L_H = \supp f$. Now, 
        we can take 
        a compact set $L_B\subseteq B$ such that
        \begin{align}
            \norm{\mu\big((a,k)(b,h)(a',k')\big) - (a,k)\cdot \mu(b,h)} < \frac \eps 2
            \label{eq:bi-exact.semi-direct product.mu}
        \end{align}
        whenever $a,a'\in K_B$, $b\in G\setminus L_B$, $k,k'\in K_H$ and $h\in L_H$.
        
        Define $\eta: G\to \Prob(G)$ by
        \[\eta(b,h) = f(h)\mu(b,h) + \big(1-f(h)\big) \nu(b,h)\]
        for $(b,h)\in G$. Set $L = \set{(b,h)\in G}{b\in L_B, h\in L_H}$. Fix $(a,k), (a',k')\in K$ and 
        $(b,k)\in G\setminus L$.  We have
        \begin{IEEEeqnarray*}{rCl}
            \norm{\eta\big((a,k)(b,h)(a',k')\big) - (a,k)\cdot \eta(b,h)} 
            &\le& f(h)\norm{\mu\big((a,k)(b,h)(a',k')\big) - (a,k)\cdot \mu(b,h)} \\
            && + \big(1-f(h)\big) \norm{\nu\big((a,k)(b,h)(a',k')\big) - (a,k)\cdot \nu(b,h)} + 2\;|f(khk') - f(h)|\\
            &\le& f(h)\norm{\mu\big((a,k)(b,h)(a',k')\big) - (a,k)\cdot \mu(b,h)} \\
            && + \big(1-f(h)\big) \norm{\nu\big((a,k)(b,h)(a',k')\big) - (a,k)\cdot \nu(b,h)} + \frac \eps 2
        \end{IEEEeqnarray*}
        We are in one of the following three cases.
        \begin{enumerate}[label=\textit{Case \arabic*.}]
            \item If $h\in H\setminus L_H$, then $f(h) = 0$ and \eqref{eq:bi-exact.semi-direct product.nu} holds.
            \item If $h\in L_H\setminus \widetilde L_H$ and $b\in B\setminus L_B$, then both \eqref{eq:bi-exact.semi-direct product.nu} and \eqref{eq:bi-exact.semi-direct product.mu} hold.
            \item If $h\in \widetilde L_H$ and $b\in B\setminus L_B$, then $f(h) = 1$ and
            \eqref{eq:bi-exact.semi-direct product.mu} holds.
        \end{enumerate}
        In all three cases, we conclude that \eqref{eq.bi-exact.semi-direct prod.TP} holds, thus proving the 
        proposition.
    \end{proof}
    \begin{remark}\label{bi-exact.semi-direct product.more general.remark on b map}
        Note that \eqref{bi-exact.semi-direct product.more general.b to infty} is equivalent with the existence
        of a map $\tilde \mu: B\to \Prob(G)$ satisfying
        \[\lim_{b\to \infty}\norm{\tilde \mu(aba') - a\cdot \tilde \mu(b)} = 0\mathand 
          \lim_{b\to \infty} \norm{\tilde \mu\big(\alpha_h(b)\big) - h\cdot \tilde \mu(b)} = 0\]
        uniformly on compact sets for $a,a'\in B$ and $h\in H$. Indeed, the restriction of a map as in 
        \eqref{bi-exact.semi-direct product.more general.b to infty} satisfies the above equations. Conversely, 
        given a map $\tilde \mu$ as above, the map $\mu: G\to \Prob(G)$ defined by $\mu(b,h) = \tilde \mu(b)$ 
        satisfies 
        \eqref{bi-exact.semi-direct product.more general.b to infty}.
    \end{remark}
    
    When the group $B$ is amenable, the previous result specializes to the corollary below.
    In the setting of countable groups, this result was proved by 
    \citeauthor{Ozawa2006} in \cite[proof of Corollary~4.5]{Ozawa2006} and \cite[Section~3]{Ozawa2009}. However, the 
    proof provided there does not carry over to the locally compact setting, since, as we explained in the 
    introduction, the characterization of class
    $\mathcal S$ in terms of a u.c.p. map $\varphi: \Cred(G)\otmin \Cred(G)\to B(L^2(G))$ satisfying $\varphi(x\otimes y) -\lambda(x)\rho(y)\in K(L^2(G))$ (see 
    \cite[Proposition~15.1.4]{Brown2008}) does not hold in this setting. 
    Also the method used in 
    \cite[Section~15.2]{Brown2008} can not be applied, since for a locally compact group $G$ the crossed product 
    $C(X)\rtimes_{r} G$ can be nuclear while $G\acts X$ is not amenable.
    
    \begin{corollary}\label{bi-exact.semi-direct product}
        Let $G = B\rtimes_\alpha H$ be a semi-direct product of lcsc groups with $B$ amenable. 
        Then $G$ is in class $\mathcal S$ if and only if $H$ is in class $\mathcal S$ and there is a Borel map 
        $\mu: B\to \Prob(H)$ such that
        \[\lim_{b\to \infty} \norm{\mu\big(\alpha_h(b)\big)- h\cdot \mu(b)} = 0\mathand \lim_{b\to \infty} \norm{\mu(aba') - \mu(b)} = 0\]
        uniformly on compact sets for $h\in H$ and $a,a'\in B$.
    \end{corollary}
    \begin{proof}
        The only if part follows immediately from 
        \cref{bi-exact.semi-direct product.more general,bi-exact.semi-direct product.more general.remark on b map}.
        Conversely, let $\mu: B\to \Prob(H)$ be a map as above. By \cref{prop(S).equivalent.continuous.technical}, we can
        assume that $\mu$ is $\norm{.}$-continuous. Let $H\acts^{\beta} B\times B$ be the diagonal action and let 
        $G_0 = (B\times B)\rtimes_\beta H$.
        We apply \cref{bi-exact.amenable subgroup.technical} to the map $\mu$ with the morphism $\pi: G_0 \to G$
        given by 
        $\pi\big(a,a',h\big) = (a,h)$ for $a,a'\in B$ and $h\in H$, the space $X = B$, the action $G_0 \acts X$ given 
        by 
        $\big(a,a',h\big)\cdot b = a\alpha_h(b)(a')^{-1}$ for $a,a',b\in B$ and $h\in H$, and the amenable subgroup 
        $K = B\subseteq G$. This yields a 
        map $\tilde \mu: B\to \Prob(G)$
        as in \cref{bi-exact.semi-direct product.more general.remark on b map}. 
        Similarly, let $\eta: H\to \Prob(H)$
        as in the definition of class $\mathcal S$. Applying \cref{bi-exact.amenable subgroup.technical} to the map 
        $\nu$ with the morphism $\pi: G\times G\to G$ given by $\pi(g,g') = g$ for $g,g'\in G$, the space $X = H$, the 
        action 
        $G\times G\acts X$ induced by left and right translation and the amenable subgroup $K = B$,
        yields a
        map $\tilde \nu: H\to \Prob(G)$ satisfying
        \[\lim_{h\to \infty} \norm{\tilde \nu (khk') - (a,k)\cdot \tilde \nu(h)} = 0\]
        uniformly on compact sets for $k,k'\in G$ and $a\in B$. Now, the map $\nu: G\to \Prob(G)$ defined by 
        $\nu(b,h)= \tilde \nu(h)$ satisfies 
        \eqref{bi-exact.semi-direct product.more general.h to infty}.
    \end{proof}
    
    We are now ready to prove \cref{bi-exact.R^2 rtimes SL_2(R)} and \cref{bi-exact.wreath product.introduction}.
    \begin{proof}[Proof of \cref{bi-exact.R^2 rtimes SL_2(R)}]
        The proof presented here is inspired by \cite{Ozawa2009}. 
        
        Write $G = \SL_2(\mathbb R)$ and $X = \mathbb R^2\setminus \{0\}$.
        Consider the compactification $\beta^GX$ of $X$ given by the spectrum of
        \[C(\beta^G X) = \set{f\in C_b(\mathbb R^2)}{\text{$\norm{A\cdot f - f}_\infty \to 0$ if $A\to 1$}},\]
        where $(A\cdot f)(x) = f(A^{-1}x)$ for $A\in G$, $f\in C_b(\mathbb R^2)$ and $x\in X$.
        Note that $\beta^GX$ is the universal compactification of $X$ on which the action of $G$ extends, i.e. every 
        continuous $G$-equivariant map $X\to Y$ to any compact space $Y$ with continuous action $G\acts Y$ extends 
        uniquely to a continuous $G$-equivariant map $\beta^GX \to Y$. 
        Also, note that for every $f\in C(\beta^GX)$, we have $\lim_{x\to \infty}|f(x+y) - f(x)| = 0$ uniformly on 
        compact sets for $y\in \mathbb R^2$.
        
        We claim that the action $G\acts \beta^GX$ is amenable. To prove this claim, consider the action of $G$ on
        the projective real line $\widehat {\mathbb R} = \mathbb R\cup\{\infty\}$ by linear fractional transformations, i.e.
        \[\begin{pmatrix}
        a & b\\
        c & d
        \end{pmatrix} \cdot t = \begin{cases}
        \frac{at+b}{ct+d} & \text{if $ct+d\neq 0$,}\\
        \infty            & \text{otherwise,}
        \end{cases}\mathand 
        \begin{pmatrix}
        a & b\\
        c & d
        \end{pmatrix} \cdot \infty = 
        \begin{cases}
        \frac{a}{c} & \text{if $c\neq 0$,}\\
        \infty      & \text{otherwise.}
        \end{cases}\]
        The stabilizer of the point $\infty \in \widehat{\mathbb R}$ is the subgroup $P\subseteq G$ of upper 
        triangular matrices. Since $P$ is solvable (and hence amenable) and 
        $G / P\to \widehat{\mathbb R}: A\mapsto A\cdot \infty$ is a homeomorphism, it follows from 
        \cite[Example~2.2.18]{Anantharaman-Delaroche2000amenable} that $G\acts \widehat{\mathbb R}\cong G / P$ is 
        amenable. Consider the map $\varphi: X \to \widehat {\mathbb R}$ defined by $\varphi(m,n) = m/n$. Since this 
        map is continuous and $G$-equivariant, it induces a $G$-equivariant extension 
        $\beta^G\varphi: \beta^GX \to \widehat{\mathbb R}$. This proves the amenability of $G\acts \beta^GX$.
        
        Now, we use \cref{bi-exact.semi-direct product} to finish the proof. Note that $G$ is in class 
        $\mathcal S$ by 
        \cite[Proof of Théorème~4.4]{Skandalis1988} (see also \cite[Proposition~7.1]{Brothier2017rigidity}). Let
        $\eta_n: \beta^GX \to \Prob(G)$ be a sequence as in the definition of an amenable action.
        By using \cref{amenable action.equivalent}, we can assume that each $\eta_n$ is $\norm{.}$-continuous. 
        We define $\tilde \eta_n: \mathbb R^2\to \Prob(G)$ by taking $\tilde \eta_n(x) = \eta_n(x)$ if $x\neq 0$ and
        $\tilde \eta_n(0)\in \Prob(G)$ arbitrary. The sequence of maps $\tilde \eta_n$ now satisfies
        \[\lim_{n\to \infty} \sup_{x\in \mathbb R^2\setminus \{0\}}\sup_{A\in K}\norm{\tilde \eta_n(Ax) - A\cdot\tilde \eta_n(x)} = 0\]
        for every compact $K\subseteq G$. Using continuity of the maps $\eta_n$, we also have
        \[\lim_{x\to \infty} \norm{\tilde \eta_n(y+x+y') - \tilde \eta_n(x)} = 0\]
        uniformly on compact sets for $y,y'\in \mathbb R^2$. 
        The map as in \cref{bi-exact.semi-direct product} can now
        be constructed by using \cref{prop(S).equivalent.net.technical} on the group 
        $\mathbb R^2\rtimes \SL_2(\mathbb R)$ and the spaces $\mathbb R^2$ and $G$ with the natural actions.
    \end{proof}
    
    We will now prove \cref{bi-exact.wreath product.introduction}. The suitable notion of wreath products
    for locally compact groups was introduced by \citeauthor{Cornulier2017wreath} in \cite{Cornulier2017wreath}.
    Let $B$ and $H$ be lcsc groups, $X$ a countable set with continuous action $H\acts X$ and $A\subseteq B$ a compact 
    open subgroup. The 
    \imph{semi-restricted power} $B^{X,A}$ is defined by
    \[B^{X,A} = \set{(b_x)_{x\in X}\in B^X}{\text{$b_x\in A$ for all but finitely many $x\in X$}}.\]
    It is a lcsc space when equipped with the topology generated by the open sets $\prod_{x\in X} C_x$ 
    where $C_x\subseteq B$ is open for every $x\in X$ and $C_x = A$ for all but finitely many $x\in X$. For 
    $b\in B^{X,A}$, we denote $\supp_A b = \set{x\in X}{b(x)\notin A}$.
    
    Denote by $\alpha$ the action of $H$ on $B^{X,A}$ by translation, i.e. $\alpha_h(b)(x) = b(h^{-1}x)$ for 
    $b\in B^{X,A}$, $h\in H$ and $x\in X$. It is easy to see that this action is continuous.
    The \imph[wreath product,semi-restricted wreath product]{(semi-restricted) wreath product} $B\wr_X^A$ is now 
    defined as
    \begin{align}
        B\wr_X^A H = B^{X,A}\rtimes_\alpha H\label{eq.wreath.definition}
    \end{align}
    equipped with the product topology. By \cite[Proposition~2.4]{Cornulier2017wreath} it is a lcsc group. 
    \cref{bi-exact.wreath product.introduction} is now a consequence of the following theorem.
    
    \begin{theorem}\label{bi-exact.wreath product}
        Let $A$, $B$, $X$ and $H$ be as above. Suppose that $B$ is non-compact and $|X|\ge 2$. Then, $B\wr_X^A H$ 
        belongs to class $\mathcal S$ if and only if $B$ is amenable, the 
        stabilizer $\Stab_H(x)$ of every point $x\in X$ is amenable and $H$ belongs to class $\mathcal S$.
    \end{theorem}
    Note that if $|X|=1$, then $B\wr_X^A H\cong B\times H$ belongs to class $\mathcal S$ if and only if both 
    factors are amenable, or one of $B$ and $H$ belongs to class $\mathcal S$ and the other is compact. 
    If $B$ is 
    compact, then by
    \cite[Lemma~7.2]{Brothier2017rigidity}, we have that $B\wr_X^A H$ belongs to class $\mathcal S$ if and only 
    if $H$ does.
    \begin{proof}[Proof of \cref{bi-exact.wreath product}]
        Suppose first that $B\wr_X^A H$ belongs to class $\mathcal S$. It follows that the subgroups $H$ and 
        $B\times B$ do. Hence, $B$ must be amenable.
        For every point $x_0\in X$ the subgroup 
        \[B\times \Stab_H(x_0)\cong \set{(b,h)\in B\wr_X^A G}{\text{$b(x) = e$ if $x\neq x_0$}}\]
        belongs to class $\mathcal S$. Since $B$ is non-compact, it follows that $\Stab_H(x_0)$ is amenable.
        
        Conversely, suppose that $H$ belongs to class $\mathcal S$ and that $B$ and all stabilizers 
        $\Stab_H(x)$ are 
        amenable. We prove that $B\wr_X^A H$ belongs to class $\mathcal S$. 
        Denote by $X = \bigcup_{i\in I} X_i$ the partition of $X$ 
        into the orbits of $H\acts X$ and fix $x_i\in X_i$ for all $i\in I$. Write $B_i = B^{X_i,A}$ and $H_i = \Stab_H(x_i)$.
        
        \paragraph{Step 1. Each $B\wr_{X_i}^A H$ belongs to class $\mathcal S$.} 
        Fix $i\in I$. To prove this step, we proceed along 
        the lines of \cite[Corollary~15.3.6]{Brown2008}. We claim that it suffices to prove the existence of a 
        continuous map 
        $\zeta_i: B_i \to M(H / H_i)^+\cong \ell_1(X_i)^+$ satisfying
        \begin{align}
            \lim_{b\to \infty} \frac{\norm{h\cdot \zeta_i(b) - \zeta_i\big(\alpha_h(b)\big)}_1}{\norm{\zeta_i(b)}_1} = 0\mathand
            \lim_{b\to \infty} \frac{\norm{\zeta_i(aba') - \zeta_i(b)}_1}{\norm{\zeta_i(b)}_1} = 0
            \label{eq:bi-exact.wreath product.zeta_i.bi-exact}
        \end{align}
        uniformly on compact sets for $h\in H$ and $a,a'\in B_i$. Indeed, suppose that $\zeta_i$ is such a map. Let 
        $H\acts^\beta B_i\times B_i$ be the diagonal action.
        We 
        first normalize $\zeta_i$ and then apply
        \cref{bi-exact.amenable subgroup.technical} to this normalized map with the
        groups $G = (B_i\times B_i)\rtimes_0 H$, $H$ and $K = H_i$, the space $X = B_i$, the morphism $\pi: G\to H$ 
        given 
        by 
        $\pi(a,a',h) = h$, and the action $G\acts X$ given by $(a,a',h)\cdot b = a\alpha_h(b)(a')^{-1}$ for 
        $a,a',b\in B_i$ 
        and $h\in H$. This yields a map $\tilde \zeta_i: B_i \to \Prob(H)$ that satisfies the conditions of 
        \cref{bi-exact.semi-direct product}.
        
        By \cite{Struble1974metrics} every lcsc group $G$ admits a continuous proper length function, i.e. a continuous 
        function $\ell: G\to \mathbb R^+$ such that $\ell(gh)\le \ell(g) + \ell(h)$ and $\ell(g) = \ell(g^{-1})$ for 
        all $g,h\in G$ and such that all the sets $\set{g\in G}{\ell(g)\le M}$ for $M > 0$ are compact. Fix such 
        continuous, proper length functions $\ell_B: B\to \mathbb R^+$ and $\ell_H: H\to \mathbb R^+$. Define the 
        function
        \[f: X_i\to \mathbb R^+: x\mapsto \inf_{\substack{h\in H\\hx_i = x}} \ell_H(h).\]
        Note that for every $M > 0$ the set $\set{x\in X_i}{f(x)\le M}$ is finite and that $f(hx)\le \ell_H(h) + f(x)$.
        Similarly, we define
        \[g: B\to \mathbb R^+: b\mapsto \inf_{a,a'\in A} \ell_B(aba'),\]
        and note that for every $M > 0$ the set $\set{b\in B}{g(b)\le M}$ is compact and that 
        $g(bb')\le g(b) + g(b') + N$, where $N = \sup_{a\in A} \ell_B(a)$. Also note that, by compactness of $A$, 
        the map $g$ is continuous.
        
        Define 
        $\zeta_i: B_i \to \ell^1(X_i)^+$ by
        \[\zeta_i(b)(x) = \begin{cases}
                            g\big(b(x)\big) + f(x) &\text{if $x\in \supp_A(b)$,}\\
                            0                   &\text{otherwise}
                        \end{cases}\]
        for $b\in B_i$ and $x\in X_i$.
        
        We prove that $\zeta_i$ satisfies \eqref{eq:bi-exact.wreath product.zeta_i.bi-exact}. Fix $h\in H$ and 
        $a,a',b\in B_i$. Denote $b' = aba'$, $S = \supp_A b$, $S' = \supp_A b'$ and $T = \supp_A a\cup \supp_A a'$.
        We have
        \[\norm{h\cdot \zeta(b) - \zeta\big(\alpha_h(b)\big)}_1
            = \sum_{x\in hS}|f(h^{-1}x) - f(x)|
            \le |S|\: \ell_H(h)\]
        and
        \begin{align*}
        \norm{\zeta(b') - \zeta(b)}_1 &= \sum_{x\in T} |\zeta(b')(x) - \zeta(b)(x)|\\
                                      &= \sum_{x\in T\cap S\cap S'} \big|g\big(b'(x)\big) - g\big(b(x)\big)\big| 
                                         + \sum_{x\in (T\cap S)\setminus S'} |g\big(b(x)\big) + f(x)|
                                         + \sum_{x\in (T\cap S')\setminus S} |g\big(b'(x)\big) + f(x)|\\
                                      &\le \sum_{x\in T} \big|g\big(b'(x)\big) - g\big(b(x)\big)\big|  + \sum_{x\in T\cap (S\triangle S')} f(x)\\
                                      &\le \sum_{x\in T} \bigg(g\big(a'(x)\big) + g\big(a(x)\big) + 2N\bigg) 
                                           + \sum_{x\in T\cap (S\triangle S')}f(x)\\
                                      &\le \norm{\zeta(a)}_1 + \norm{\zeta(a')}_1 + 2N\;|T|
        \end{align*}
        where we used in the third step that $g(b) = 0$ whenever $b\in A$.
        
        So, it suffices to prove that 
        \[\lim_{b\to \infty} \norm{\zeta_i(b)}_1 = +\infty\mathand 
          \lim_{b\to \infty} \frac{|\supp_A b|}{\norm{\zeta_i(b)}_1} = 0.\]
        To prove the first, suppose that $\norm{\zeta(b)}_1\le M$ for some $M > 0$. Then, 
        $f(x)\le M$ and $g\big(b(x)\big)\le M$ for every $x\in\supp_A(b)$. Hence,
        \[b\in C = \prod_{x\in X_i} C_x\]
        where $C_x = \set{b\in B}{g(b)\le M}$ for $x\in F = \set{x\in X}{f(x)\le M}$ and $C_x = A$ otherwise. 
        Since $F$ is finite and each $C_x$ is compact, it follows that $C$ is compact, which in turn implies the claim.
        
        To prove that ${|\supp_A b|} / {\norm{\zeta_i(b)}_1}\to 0$ if $b\to \infty$. Suppose that 
        ${|\supp_A b|} / {\norm{\zeta_i(b)}_1} \ge \delta$ for some $b\in B$ and $\delta > 0$. Denote 
        $D = \set{x\in X_i}{f(x)\le 2/\delta}$. Then,
        \[\frac{2}{\delta} \big(|\supp_A b| - |D|\big) \le \frac{2}{\delta} |\supp_A b\setminus D| \le \norm{\zeta_i(b)}_1 \le \frac{1}\delta |\supp_A b|\]
        and thus $|\supp_A b|\le 2|D|$. It follows that $\norm{\zeta_i(b)}_1 \le \frac{2}\delta |D|$. But, by the 
        previous, the set
        \[\set{b\in B}{\norm{\zeta(b)}_1 \le \frac 2\delta |D|}\]
        is compact and hence so is $\set{b\in B}{{|\supp_A b|} / {\norm{\zeta_i(b)}_1}\ge \delta}$.
        
        \paragraph{Step 2. Construction of maps $\xi_i: B_i\to \Prob(H)$ satisfying \eqref{eq:bi-exact.wreath product.xi_i} below.} 
        Fix $i\in I$, $\eps > 0$ and a compact $K\subseteq H$. In this step, we construct a Borel 
        map $\xi_i: B_i\to \Prob(H)$
        such that
        \begin{align}
            \norm{\xi_i\big(\alpha_h(b)\big) - h\cdot\xi_i(b)}\le \eps\mathand \xi_i(aba') = \xi_i(b)\label{eq:bi-exact.wreath product.xi_i}
        \end{align}
        for all $b\in B_i\setminus A^{X_i}$, all $h\in K$ and all $a,a'\in A^{X_i}$. Note that the difference with
        the previous step is that we want the map $\xi_i$ to satisfy \eqref{eq:bi-exact.wreath product.xi_i}
        for all $b\in B_i\setminus A^{X_i}$, instead of $b\in B_i\setminus L$ for $L$ some (possibly large) compact set.
        
        Since $H_i$ is amenable, the action $H\acts H / H_i$ is amenable. Indeed, let $(\nu_n)_n$ be a 
        sequence in 
        $\Prob(H_i)$ such that $\norm{h\cdot \nu_n - \nu_n}\to 0$ uniformly on compact sets for $h\in H_i$ when 
        $n\to \infty$. Fix a 
        cross section $\sigma: H / H_i\to H$ for the quotient map $p: H\to H / H_i$. 
        Then, the sequence of maps $\eta_n: H / H_i\to \Prob(H)$ defined by
        \[\eta_n(hH) = \sigma(hH)\cdot \nu_n\]
        satisfies
        \begin{align*}
            \lim_{n\to \infty} \norm{h\cdot \eta_n(h'H) - \eta_n(hh'H)} = \lim_{n\to \infty}\norm{\sigma(hh')^{-1}h\sigma(h'H)\cdot \nu_n - \nu_n} = 0
        \end{align*}
        uniformly on compact sets for $h\in H$ and $h'H\in H/H_i$. 
        
        By \cref{amenable action.uniform on space}, it follows that we can take a sequence of maps such that the 
        convergence holds uniformly on the whole of $H / H_i$. Hence, 
        identifying $X_i \cong H / H_i$, we find a map
        $\mu: X_i \to \Prob(H)$ such that
        \[\norm{h\cdot \mu(x) - \mu(hx)} < \eps\]
        for every $h\in K$ and every $x\in X_i$. Now, define $\xi_i: B_i\to \Prob(H)$ by
        \[\xi_i(b) = \frac{1}{|\supp_A b|}\sum_{x\in \supp_A b} \mu(x)\]
        for $b\in B_i\setminus A^{X_i}$. For $b\in B_i$, set $\xi_i(b) = \mu_0$, where $\mu_0\in \Prob(H)$ is arbitrary.
        One easily checks that $\xi_i$ satisfies \eqref{eq:bi-exact.wreath product.xi_i}.
        
        \paragraph{Step 3. $B\wr_X^A H$ is bi-exact.} 
        Take $\eps > 0$, a compact $C\subseteq B^{X,A}$ 
        and a compact $K\subseteq H$. By \cref{prop(S).equivalent.net.technical} and 
        \cref{bi-exact.semi-direct product}, it suffices to prove that there exists a compact $D\subseteq B^{X,A}$ and 
        a Borel map $\zeta: B^{X,A}\to \Prob(H)$ such that
        \begin{align}
            \norm{h\cdot \zeta(b) - \zeta\big(\alpha_h(b)\big)} \le \eps \mathand \norm{\zeta(aba') - \zeta(b)}\le \eps
            \label{eq:bi-exact.wreath product.zeta}
        \end{align}
        for all $h\in K$, $a,a'\in C$ and $b\in B^{X,A}\setminus D$. 
        
        By definition of the topology on the semi-restricted 
        product $B^{X,A}$, we can take compact sets $C_i\subseteq B_i$ for $i\in I$ such that
        \[C\subseteq \prod_{i\in I} C_i\]
        and such that $C_i = A^{X_i}$ for all but finitely many $i\in I$.
        Take $i_1,\dots, i_n\in I$ such that $C_i = A^{X_i}$ whenever $i\neq i_1,\dots, i_n$. 
        
        For $i = i_1, \dots, i_n$, the fact that $B\wr_{X_i}^A H$ belongs to class $\mathcal S$, allows us to 
        take a compact $D_i\subseteq B_i$ and a Borel map $\zeta_i: B_i\to \Prob(H)$ such that
        \[\norm{h\cdot \zeta_i(b) - \zeta_i\big(\alpha_h(b)\big)} 
          \le \eps \mathand \norm{\zeta_i(aba') - \zeta_i(b)}\le \eps\]
        for $h\in K$, $a,a'\in C_i$ and $b\in B_i\setminus D_i$. By enlarging $D_i$, we can assume that 
        $A^{X_i}\subseteq D_i$ and $C_i^{-1}A^{X_i}C_i^{-1}\subseteq D_i$.
        For $i\neq i_1,\dots,i_n$, we take for $\zeta_i: B_i \to \Prob(H)$ the map $\xi_i$ from step 2 and set 
        $D_i = A^{X_i}$. 
        
        For $b\in B^{X,A}$ and $i\in I$, we denote by $b_i\in B^{X_i,A}$ the restriction of $b$ to $X_i$. We also
        denote $I_b = \set{i\in I}{b_i\notin A^{X_i}}$. Define $\zeta: B^{X,A}\to \Prob(H)$
        by
        \[\zeta_i(b) = \frac{1}{|I_b|} \sum_{i\in I_b} \zeta_i(b_i) \]
        for $b\in B^{X,A}\setminus A^X$ and $\zeta_i(b) = \delta_e$ for $b\in A^X$. One easily checks that 
        \eqref{eq:bi-exact.wreath product.zeta} holds for $D = \prod_{i\in I} D_i$, since $I_b = I_{aba'}$ for 
        $b\in B^{X,A}\setminus D$ and $a,a'\in C$.
    \end{proof}
    
    For completeness, we also include a proof of the following fact mentioned in the introduction. It is a locally
    compact version of a result mentioned in \cite{Ozawa2006} in the countable setting.
    \begin{proposition}\label{bi-exact.inner amenable}
        A lcsc group $G$ that is inner amenable at infinity belongs to class $\mathcal S$ if and only if
        $G$ is amenable.
    \end{proposition}
    \begin{proof}
        If $G$ is amenable, the result is immediate. Conversely, suppose that $G$ is in class $\mathcal S$. Let
        $\eta: G\to \Prob(G)$ be a map as in the definition. Define the map $\eta_*: C_b(G)\to C_b(G)$ by
        \[(\eta_*f)(g) = \int_G f\dif\eta(g).\]
        It is easy to prove that $\eta_*(\lambda_g f) - \lambda_g \eta_*(f)\in C_0(G)$ and 
        $\eta_*(f) - \rho_g\eta_*(f)\in C_0(G)$ for all $f\in C_b(G)$ and $g\in G$.
        
        Since $G$ is inner amenable at infinity, we can take a state $m: C_b(G)\to \mathbb C$ that is invariant under conjugation and such that $m(f) = 0$ for all $f\in C_0(G)$. Then,
        \[m\circ \eta_*(\lambda_g f) = m\big(\lambda_g\eta_*(f)\big) = m\big(\rho_g\lambda_g \eta_*(f)\big) = m\circ \eta_*(f).\]
        Hence, $m\circ \eta_*$ is a left-invariant mean on $C_b(G)$.
    \end{proof}
    
    \section{Class \texorpdfstring{$\mathcal S$}{𝓢} is closed under measure equivalence}
    In this section, we prove \cref{bi-exact.measure equivalent}. As mentioned in the introduction, exactness is 
    preserved under measure equivalence. So, it suffices to prove that property (S) (i.e. the existence of a map 
    $\eta: G\to \Prob(G)$ satisfying \eqref{eq:prop (S)}) is a measure equivalence invariant. In order to prove that, 
    we will use the characterization of 
    measure equivalence in terms of cross section equivalence relations \cite[Theorem~A]{KKR2018non-unimodular}
    and introduce a notion of property (S) for these relations.
    
    Recall that a countable, Borel equivalence relation $\mathcal R$ on a standard probability space 
    $(X,\mu)$ is an
    equivalence relation on $X$ such that $\mathcal R\subseteq X\times X$ is a Borel subset and such
    that all orbits are countable. We say that $\mathcal R$ is non-singular for the measure $\mu$ if $\mu(E) = 0$
    implies that $\mu([E]_{\mathcal R}) = 0$ for all measurable $E\subseteq X$. Here, 
    $[E]_{\mathcal R} = \set{x\in X}{\exists y\in E: x\sim_{\mathcal R} y}$.
    We say that $\mathcal R$ is \imph{ergodic} if $E = [E]_{\mathcal R}$ implies that $\mu(E) = 0$ or $\mu(E)=1$.
    We denote $\mathcal R^{(2)} = \set{(x,y,z)}{x\sim_{\mathcal R}y\sim_{\mathcal R}z}$. Note that 
    $\mathcal R^{(2)}\subseteq X\times X\times X$ is Borel.
    
    A Borel subset $\mathcal W\subseteq \mathcal R$ is called \emph{bounded} if the number of elements in its
    sections is bounded, i.e. if there exists a $C > 0$ such that
    \[|{_x\mathcal W}| = |\set{y\in X}{(x,y)\in \mathcal W}|<C \mathand
      |{\mathcal W_y}| = |\set{x\in X}{(x,y)\in \mathcal W}| < C\]
    for a.e. $x,y\in X$. We say that $\mathcal W$ is \imph{locally bounded} if for every $\eps > 0$, there exists a 
    Borel
    subset $E\subseteq X$ with $\mu(X\setminus E)\le \eps$ such that 
    $\mathcal W\cap (E\times E)$ is bounded.
    
    The \imph{full group} $[\mathcal R]$ is the group of all Borel automorphisms $\varphi: X\to X$, identified up to almost everywhere equality, such that
    $\grph \varphi = \{(\varphi(x),x)\}_{x\in X}$ is contained in $\mathcal R$. The \imph{full pseudo group} 
    $[[\mathcal R]]$ is the set of all partial Borel isomorphisms $\varphi: A\to B$ for Borel sets $A,B\subseteq X$
    whose graph is contained in $\mathcal R$. Again, these partial isomorphisms are identified up to almost everywhere
    equality. Every bounded Borel subset $\mathcal W\subseteq \mathcal R$ can be written as a finite union of graphs
    of elements in $[[\mathcal R]]$. For more information about countable equivalence relations, see for instance 
    \cite{Feldman1977I}.
    
    Let $G$ be a lcsc group and $G\acts (X,\mu)$ a probability measure preserving (pmp) action. We say that the action 
    $G\acts (X,\mu)$ is \imph{essentially free} if the set
    \[\set{x\in X}{\exists g\in G: gx = x}\]
    is a null set. Note that this set is Borel by \cite[Lemma~10]{Meesschaert2013}.
    
    The notion of a cross section equivalence relation was originally introduced by \citeauthor{Forrest1974} in 
    \cite{Forrest1974}. A more recent, self-contained treatment for unimodular groups can be found in 
    \cite{Kyed2015}.
    Given an essentially free pmp action $G\acts (X,\mu)$ on a 
    standard probability space, a 
    \imph{cross section} is a Borel subset $X_1\subseteq X$ with the following two properties.
    \begin{conditions}
        \item There exists a neighborhood $\mathcal U\subseteq G$ of identity such that the action map $\mathcal U\times X_1\to X: (g,x)\mapsto gx$ is injective.
        \item The subset $G\cdot X_1 \subseteq X$ is conull.
    \end{conditions}
    By \cite[Theorem~4.2]{Forrest1974} such a cross section always exists.
    Note that the first condition implies that the action map $\theta: G\times X_1\to X: (g,x)\mapsto gx$ is 
    countable-to-one
    and hence maps Borel sets to Borel sets. In particular, the set $G\cdot X_1$ in the second condition is Borel.
    
    By removing a $G$-invariant null set from $X$, we can always assume that $G\cdot X_1 = X$ and that $G\acts X$ is 
    really free. Hence, by \cite[18.10 and 18.14]{Kechris1995}, we can take a Borel map that is a right inverse of the 
    map
    $G\times X_1\to X: (g,x)\mapsto gx$. This yields Borel maps $\pi: X\to X_1$ and $\gamma: X\to G$ 
    such that $x = \gamma(x)\cdot \pi(x)$ for all $x\in X$. Similarly, the map 
    $G\times X\to X\times X: (g,x)\mapsto (gx,x)$ is injective and hence has a Borel image, which we denote by 
    $\mathcal R_G$, and an inverse that is Borel. This yields a Borel map $\omega: \mathcal R_G \to G$ satisfying
    $\omega(x,y)y = x$ for $y\in G\cdot x$. Moreover, $\omega$ is a 1-cocycle in the sense that 
    $\omega(x,y)\omega(y,z) = \omega(x,z)$ for all 
    $y,z\in G\cdot x$.
    
    The \imph{cross section equivalence relation} associated to $X_1$ is defined by
    \[\mathcal R = \mathcal R_G\cap (X_1\times X_1) = \set{(x,y)\in X_1\times X_1}{y\in G\cdot X_1}.\]
    The measurable space $X_1$ admits a unique probability measure 
    $\mu_1$ and a unique number $0 < \covol (X_1) < +\infty$ such that
    \begin{align}
        (\lambda_G\otimes \mu_1)(\mathcal W) = \covol(X_1) \int_X |\mathcal W\cap \theta^{-1}(x)| \dif \mu(x)
        \label{eq:cross section.measure}
    \end{align}
    for all measurable $\mathcal W\subseteq G\times X_1$.
    The relation $\mathcal R$ is a non-singular, countable, Borel equivalence relation for this probability measure 
    $\mu_1$.
    
    \pagebreak[3]
    We will use the following easy lemma throughout the rest of this section.
    \begin{lemma}\label{cer.cross section.easy}
        Let $G$ be a lcsc group and $G\acts (X,\mu)$ an essentially free, pmp action. Let $X_1\subseteq X$ be a 
        cross section and $\mathcal R$ the associated cross section equivalence relation. Then,
        \begin{properties}
            \item If $K\subseteq G$ is compact, then the set $\mathcal W = \set{(x,y)\in \mathcal R}{\omega(x,y)\in K}$ 
            is a bounded 
            subset of $\mathcal R$.\label{cer.cross section.easy.compact -> bdd}
            \item If $\mathcal W\subseteq R$ is a locally bounded set and $\eps > 0$, then there exists a Borel subset 
            $E\subseteq X_1$ with $\nu(E) < \eps$ such that $\omega\big(\mathcal W\cap (E\times E) \big)$ is relatively 
            compact.\label{cer.cross section.easy.bdd -> compact}
        \end{properties}
    \end{lemma}
    \begin{proof}
        Statement \ref{cer.cross section.easy.compact -> bdd} follows easily from the fact that there is a neighborhood 
        of the unit $e\in G$ for which the map 
        $\mathcal U\times X_1\to X: (g,x)\mapsto gx$ is injective.
        
        Since every bounded Borel subset can be written as a finite union of graphs of elements in $[[\mathcal R]]$, it
        suffices to prove \ref{cer.cross section.easy.bdd -> compact} for $\grph(\varphi)$ with 
        $\varphi\in [[\mathcal R]]$, but this can be done easily by taking 
        $E = \set{x\in X}{\omega\big(\alpha(x),x\big)\in K}$ for $K$ a compact set that is large enough.
    \end{proof}
    
    We define property (S) on the level of non-singular, countable, Borel equivalence relations as follows.
    \begin{definition}\label{me.prop(S).cer}\label{cer.prop(S).def}
        Let $\mathcal R$ be a non-singular, countable, Borel equivalence relation on a standard measure space $(X,\mu)$.
        We say that $\mathcal R$ has \imph{property~(S)} if there exists a Borel map 
        $\eta:\mathcal R^{(2)}\to \mathbb C$ such that
        \[\sum_{\substack {z\in X\\z\sim x}} \eta(x,y,z) = 1\]
        for a.e. $(x,y)\in \mathcal R$ and such that for all $\eps > 0$ and $\varphi,\psi\in [\mathcal R]$ the set
        \begin{align}
            \set[3]{(x,y)\in \mathcal R}{\sum_{\substack {z\in X\\z\sim x}} |\eta(x,y,z) - \eta(\varphi(x),\psi(y),z)|\ge \eps}
            \label{eq:prop(S).cer}
        \end{align}
        is locally bounded.
    \end{definition}
    \begin{remark}
        The map $\eta$ above can be viewed as a map assigning to all $(x,y)\in \mathcal R$ a probability measure on the
        orbit of $y$ such that for all $\eps > 0$ and all $\varphi, \psi\in [\mathcal R]$ the set
        \begin{align}
            \set{(x,y)\in \mathcal R}{\norm{\eta(\varphi(x),\psi(y) - \eta(x,y)}_1\ge \eps}\label{eq:prop(S).cer.prop}
        \end{align}
        is locally bounded.
    \end{remark}
    
    We prove first that this notion of property (S) is stable under restrictions and amplifications of ergodic, 
    countable equivalence relation.
    \begin{lemma}\label{me.prop(S).cer.restriction}
        Let $\mathcal R$ be a countable, ergodic, non-singular equivalence relation on some standard probability space 
        $(X,\mu)$ and let $X_0\subseteq X$ be a Borel subset with positive measure. 
        Then, $\mathcal R$ has property (S) if and only if the restriction 
        $\mathcal R_0 = \mathcal R\cap (X_0\times X_0)$ has property (S).
    \end{lemma}
    \begin{proof}
        Since $\mathcal R$ is ergodic, we can take a partition $Y = \bigcup_i Y_i$ and Borel isometries
        $\varphi_i \in [\mathcal R]$ such that $\varphi_i(Y_i)\subseteq Y_0$. 
        
        Suppose first that $\mathcal R_0$ has property (S) and let $\eta_0$ be as in \cref{me.prop(S).cer}. We extend 
        $\eta_0$ to a map $\eta$ on $\mathcal R$ by setting
        \[\eta(x,y) = \eta_0\big(\varphi_i(x), \varphi_j(y)\big)\]
        for every $(x,y)\in \mathcal R$ with $x\in Y_i$ and $y\in Y_j$.
        It is straightforward to check that $\eta$ satisfies \eqref{eq:prop(S).cer} for every $\eps> 0$ and 
        $\varphi,\psi\in [\mathcal R]$.
        
        Conversely, suppose that $\mathcal R$ has property (S). Let $\eta$ be a map as in the definition. Define for 
        $(x,y)\in \mathcal R_0$ a probability measure on the $\mathcal R_0$-orbit of $y$ by setting
        \[\eta_0(x,y)(z) = \sum_{\substack{i\in I\\z\in \varphi_i(Y_i)}} \eta(x,y)\big(\varphi_i^{-1}(z)\big)\]
        whenever $(x,y)\in \mathcal R_0$.
        Clearly, $\eta_0$ satisfies 
        \eqref{eq:prop(S).cer} for every $\eps > 0$ and every $\varphi,\psi\in [\mathcal R_0]$
    \end{proof}
    
    Now, we prove that the above notion of property (S) is compatible with taking cross section equivalence relations.
    \begin{proposition}\label{me.prop(S).cross section}
        Let $G$ be a lcsc group and $G\acts (X,\mu)$ an essentially free, ergodic, pmp 
        action. Let 
        $X_1\subseteq X$ be a cross section and $\mathcal R$ the associated cross section equivalence relation. Then,
        $G$ has property (S) if and only if $\mathcal R$ has property (S).
    \end{proposition}
    \begin{proof}
        As before, we fix Borel maps $\gamma: X\to G$ and $\pi: X\to X_1$ such that $x = \gamma(x)\cdot \pi(x)$ for 
        a.e. $x\in X$.
        First, assume that $G$ has property (S). Let $\eta: G\to \Prob(G)$ be a map satisfying \eqref{eq:prop (S)}. 
        Define for each $x\in X$ a map
        $$\pi_x : G \to X_1 : g\mapsto \pi(g^{-1} x).$$
        Note that $\pi_x$ is a Borel map from $G$ to the $\mathcal R$-orbit of $\pi(x)$. 
        We define the map $\eta'$ as in \cref{cer.prop(S).def} by
        \[\eta'(x,y) = (\pi_x)_*\eta\big(\omega(x,y)\big)\]
        for $(x,y)\in \mathcal R$.
        Note that indeed every $\eta'(x,y)$ is a probability measure on the 
        $\mathcal R$-orbit of $x$. 
        
        To prove that $\eta'$ satisfies \eqref{eq:prop(S).cer.prop}, fix $\eps,\delta > 0$ and 
        $\varphi,\psi \in [\mathcal R]$. 
        It suffices to find a Borel set 
        $E \subset X_1$ with $\mu_1(X_1 \setminus E) < \delta$ such that the set
        \begin{equation}\label{eq.mygoal}
            \set{ (x,y) \in \mathcal R \cap (E \times E)}{\norm{\eta'\big(\varphi(x),\psi(y)\big) - \eta'(x,y)}_1 \geq \eps}
        \end{equation}
        is bounded. 
        
        By \cref{cer.cross section.easy}, we find a compact set $K\subseteq G$ and a measurable $E\subseteq X_1$ with
        $\mu_1(X_1\setminus E) < \delta$ such that 
        $\omega\big(\varphi(x),x\big)\in K$ and $\omega\big(y,\psi(y)\big)\in K$ for all $x,y\in E$. 
        Take a compact set $L \subset G$ such that $\|\eta(gkh) - g \cdot \eta(k)\|_1 < \eps$ for all $g,h \in K$ and 
        all $k \in G \setminus L$. We claim that
        \begin{equation}\label{eq.myclaim}
            \norm{\eta'\big(\varphi(x),\psi(y)\big) - \eta'(x,y)}_1 < \eps
        \end{equation}
        whenever $(x,y)\in \mathcal R\cap (E \times E)$ and $(x,y)\in G\setminus L$.
        Assuming the claim is true, the set \eqref{eq.mygoal} is contained in the set of all $(x,y) \in \mathcal R$ 
        with $\omega(x,y) \in L$ which is bounded by \cref{cer.cross section.easy}. 
        To prove \eqref{eq.myclaim}, fix $(x,y) \in \mathcal R \cap (E\times E)$ with 
        $\omega(x,y) \in G \setminus L$. We have
        \begin{align*}
            \norm{\eta'\big(\varphi(x),\varphi(y)\big) - \eta'(x,y)}_1 
                &= \norm{ (\pi_{\varphi(x)})_* \eta\big(\omega(\varphi(x),\psi(y) )\big) 
                      - (\pi_x)_* \eta\big( \omega(x,y) \big)}_1
        \end{align*}
        Now, $\pi_x(g) = \pi_{\varphi(x)}\big(\omega(\varphi(x),x)g \big)$
        and hence
        \[(\pi_x)_*\big(\eta\big(\omega(x,y)\big) \big)
           = (\pi_{\varphi(x)})_*\big(\omega\big(\varphi(x),x\big)\cdot \eta\big(\omega(x,y)\big)\big)\]
        which yields that
        \begin{align*}
        \norm{\eta'\big(\varphi(x),\psi(y)\big) - \eta'(x,y)}_1 
            &= \norm{\eta\big(\omega(\varphi(x),\psi(y))\big) - \omega\big(\varphi(x),x\big)\cdot \eta\big(\omega(x,y)\big)}_1
            < \eps.
        \end{align*}
        where we used the identity
        \[\omega\big(\varphi(x),\psi(x)\big) = \omega\big(\varphi(x),x\big) \omega(x,y)\omega\big(y,\psi(y)\big)\] 
        and the assumption that $\omega\big(\varphi(x),x\big), \omega\big(y,\psi(y)\big)\in K$ and 
        $\omega(x,y)\in G\setminus L$.
        Hence, \eqref{eq.myclaim} is proved.
        
        Conversely, assume that $\mathcal R$ has property (S) and let $\eta$ be a map as in the definition. Choose
        an arbitrary $\xi\in \Prob(G)$ and define
        \[\eta': G\to \Prob(G): g\mapsto \int_X \Big(\sum_{\substack{z\in X_1\\ z\sim \pi(x)}} \eta\big( \pi(g x), \pi(x),z \big)  \omega(gx,z)\cdot \xi  \Big)\dif \mu(x).\]
        We prove that $\eta'$ is a satisfies \eqref{eq:prop (S)}. To motivate the arbitrary choice of $\xi$, note that 
        whenever $\eta'$ satisfies \eqref{eq:prop (S)}, so does the map $g\mapsto \eta'(g)*\xi$, where $\eta'(g)*\xi$
        denotes the convolution product of $\eta'(g),\xi\in \Prob(G)$.
        
        Fix a symmetric, compact neighborhood $K$ of the unit $e$ in $G$ and an $\eps > 0$. Take a compact, 
        symmetric subset $L\subseteq G$ such that $F = \gamma^{-1}(L)$ satisfies 
        $\mu(F) \ge 1-\eps$. Denote $\kappa = \lambda_G(L) / \covol (X_1)$. By \cref{cer.cross section.easy}, the set
        \[\mathcal W = \set{(x,y)\in \mathcal R}{\omega(x,y)\in LKL}\]
        is bounded Borel. Writing $\mathcal W$ as a union of finitely many elements of $[[\mathcal R]]$ and using 
        \eqref{eq:prop(S).cer.prop}, we see that the set
        \[\mathcal V = \set{(x,y)\in \mathcal R}{ \exists(x,x'), (y,y')\in \mathcal W, \norm{\eta(x',y')- \eta(x,y)}_1\ge \eps }\]
        is locally bounded. Denoting $\delta =\eps / \kappa$ and using \cref{cer.cross section.easy}, we can find a 
        compact set $C\subseteq G$ and a measurable $E\subseteq X_1$ and with 
        $\mu_1(E)\ge 1-\delta$ such that $\omega\big(\mathcal V\cap (E\times E)\big) \subseteq C$. We conclude that
        \begin{align}
            \norm{\eta(x',y') - \eta(x,y)}_1 < \eps\label{eq:me.prop(S).cer => group.given}
        \end{align}
        whenever $(x,y)\in \mathcal R\cap (E\times E)$ with $(x,x')\in \mathcal W$, 
        $(y,y')\in \mathcal W$ and $\omega(x,y)\in G\setminus C$.
        
        Denote $D = LCL$. We conclude the proposition by proving that
        \begin{align}
            \norm{\eta'(gkh) - g\cdot \eta'(k)} < 4\kappa\delta + 9\eps = 13\eps\label{eq:me.prop(S).cer => group.final}
        \end{align}
        for all $g,h\in K$ and $k\in G\setminus D$. So, fix $g,h\in K$ and $k\in G\setminus D$. 
        Applying the change of variables $x\mapsto h^{-1}x$ and using that
        $\omega(gkx,z) = g\omega(kx,z)$, we find that
        \[\eta'(gkh) = g\cdot \Bigg(\int_X \Big( \sum_{\substack{z\in X_1\\ z\sim \pi(x)}} \eta\big( \pi(gk x), \pi(h^{-1}x),z \big)  \omega(kx,z)\cdot \xi \Big)\dif \mu(x) \Bigg)\]
        and hence
        \[\norm{\eta'(gkh) - g\cdot \eta'(k) } \le \int_X \norm{\eta\big(\pi(gkx), \pi(h^{-1}x)\big) - \eta\big(\pi(kx), \pi(x) \big)  }_1\dif \mu(x).\]
        Since $g,h^{-1}\in K$, we have that $\big(\pi(gkx), \pi(kx)\big)\in \mathcal W$ 
        and 
        $\big(\pi(h^{-1}x), \pi(x) \big)\in \mathcal W$ whenever $x\in X$ is such that 
        $gkx,h^{-1}x,kx,x\in F = \gamma^{-1}(L)$. Moreover, for such 
        an
        $x$
        we also have $\omega\big(\pi(kx), \pi(x)\big)\in LkL\subseteq G\setminus C$.         
        Hence, by 
        \eqref{eq:me.prop(S).cer => group.given} we have that
        \begin{align}
            \norm{\eta\big(\pi(gkx), \pi(h^{-1}x)\big) - \eta\big(\pi(kx), \pi(x) \big)}_1 <\eps\label{eq:me.prop(S).cer => group.goal}
        \end{align}
        whenever $gkx,h^{-1}x,kx,x\in F$, $\pi(x)\in E$ and $\pi(kx)\in E$.
        
        Since $\mu(F) \ge 1-\eps$, we can find a measurable set $F'$ with 
        $\mu(F')\ge 1-4\eps$ such that $gkx,h^{-1}x,kx,x\in F$ for every $x\in F'$. Moreover,
        the map $\theta: G\times X_1\to X$ is injective on the image $A$ of the map 
        $x\mapsto \big(\gamma(x), \pi(x)\big)$. 
        Hence by \eqref{eq:cross section.measure}, we have that
        $\covol(X_1)\;\mu\big(\theta(\mathcal U)\big) = (\lambda_G\otimes \mu_1)(\mathcal U)$ for all 
        $\mathcal U\subseteq A$. 
        It follows 
        that for measurable $S\subseteq X_1$, we have that
        \[\mu\big(\pi^{-1}(S)\cap F\big) = \covol(X_1)^{-1} (\lambda_G\times \mu_1)\big(A\cap (L\times S)\big)\le \frac{\lambda_G(L)}{\covol(X_1)} \mu_1(S) = \kappa \mu_1(S).\]
        Applying this to 
        $\pi^{-1}(X_1\setminus E)\cap F$ and using the definition $F'$ above, we conclude 
        that \eqref{eq:me.prop(S).cer => group.goal} holds on a set whose complement has at most measure 
        $4\eps+2\kappa\delta$
        and hence that \eqref{eq:me.prop(S).cer => group.final} holds.
    \end{proof}
    
    The proof of \cref{bi-exact.measure equivalent} is now easy.
    \begin{proof}[Proof of \cref{bi-exact.measure equivalent}]
        Let $G$ be a lcsc group in class $\mathcal S$ and let $H$ be a lcsc group that is measure equivalent to 
        $G$. As mentioned in the introduction, we have that $H$ is exact. 
        Indeed, by \cite[Corollary~2.9]{DL2015propA} and 
        \cite[Theorem~A]{BCL2016exactness} $G$ is exact if and only if the proper metric space $(G,d)$ has 
        property (A) in the sense of Roe, where $d$ is any proper left-invariant metric that implements the topology on 
        $G$, and by \cite[Theorem~0.1~(6)]{Deprez2014permanence}, property A is a measure equivalence invariant.
        
        By 
        \cite[Theorem~A]{KKR2018non-unimodular} and \cite[Theorem~A]{Koivisto2017measure}, $G$ and $H$ admit free, 
        ergodic, probability measure preserving actions $G\acts (X,\mu)$ and $H\acts (Y,\nu)$ with cross sections 
        $X_1\subseteq X$, $Y_1\subseteq Y$ and cross section equivalence relations $\mathcal R$ and $\mathcal T$ 
        respectively such that $\mathcal R$ is stably isomorphic to $\mathcal T$. But, by
        \cref{me.prop(S).cross section}, the relation $\mathcal R$ (resp. $\mathcal T$) has property (S) if and only if 
        $G$ (resp. 
        $H$) has and by \cref{me.prop(S).cer.restriction}, $\mathcal R$ has property (S) if and only if 
        $\mathcal T$ has.
    \end{proof}
    
    \section{Class \texorpdfstring{$\mathcal S$}{𝓢} and unique prime factorization}
    In \cite{Houdayer2017unique}, \citeauthor{Houdayer2017unique} introduce the following property.
    \begin{definition}
        Let $(M,\mathcal H,J,\mathfrak P)$ be a von Neumann algebra in standard form. We say that $M$ satisfies the 
        \imph{strong condition (AO)} if there exist \C*-algebras $A\subseteq M$ and $\mathcal C\subseteq B(\mathcal H)$ 
        such that
        \begin{itemize}
            \item $A$ is exact and $\sigma$-weakly dense in $M$,
            \item $\mathcal C$ is nuclear and contains $A$,
            \item all commutators $[c,JaJ]$ for $c\in \mathcal C$ and $a\in A$ belong to the compact operators $K(\mathcal H)$.
        \end{itemize}
    \end{definition}
    Note that the definition in \cite[Definition~2.6]{Houdayer2017unique} also requires $A$ and $\mathcal C$ to be 
    unital. However, by \cite[Proposition~2.2.1 and Proposition~2.2.4]{Brown2008} this requirement is not essential.

    In \cite[Theorems~A~and~B]{Houdayer2017unique}, \citeauthor{Houdayer2017unique} provide unique factorization 
    theorems for nonamenable factors satisfying strong condition (AO). 
    \Cref{bi-exact.UPF.source and target,bi-exact.UPF.source} now follow immediately by combining these theorems with
    the following result.
    
    \begin{proposition}
        Let $G$ be a lcsc group in class $\mathcal S$, then its group von Neumann algebra $L(G)$ satisfies strong condition (AO).
    \end{proposition}
    \begin{proof}
        Recall that $L(G)$ is in standard form on $L^2(G)$. The anti-unitary operator $J$ is given by
        \[(J\xi)(t) = \delta_G(t)^{-1/2}\overline {\xi(t^{-1})},\]
        where $\delta_G$ denotes the modular function of $G$. 
        The action of an element $\lambda(f)\in L(G)$ for $f\in C_c(G)$ is given by
        \[\big(\lambda(f)\xi\big)(s) = \int_G f(t)\xi(t^{-1}s)\dif t.\]
        Straightforward calculation yields
        \[\big(J\lambda(f)J\xi\big)(s) = \int_G \overline{f(t)}\delta_G(t)^{1/2}\xi(st)\dif t.\]
        
        Let $A = \Cred(G)$ be the reduced group \C*-algebra of $G$. Then, obviously $A$ is exact and $\sigma$-weakly 
        dense in $L(G)$.
        By \cref{bi-exact.equivalent} and \cite[Theorem~5.3]{Anantharaman-Delaroche2000}, the 
        algebra $C(h^uG)\rtimes G$ is nuclear. Now, the inclusion 
        $C(h^u G)\subseteq C_b^{u}(G)\hookrightarrow B(L^2(G))$ together with the unitary representation 
        $g\mapsto \lambda_g$ induces a $^*$-morphism $\pi: C(h^u G)\rtimes G\to B(L^2(G))$. Let $\mathcal C$ be 
        the
        image of this $^*$-morphism. The algebra $\mathcal C$ is nuclear as a quotient of a nuclear \C*-algebra, and 
        obviously contains $A$. Note that $C_c\big(G,C(h^uG)\big)$ is a dense subalgebra in $C(h^uG)\rtimes G$. 
        Identifying an
        element $h\in C_c\big(G,C(h^uG)\big)\subseteq C(h^uG)\rtimes G$ with a function on $G\times G$ that is 
        compactly supported in the first 
        component, we get that the action $\pi(h)$ on a $\xi \in L^2(G)$ is given by
        \[\big(\pi(h)\xi\big)(s) = \int_G h(t,s) \xi(t^{-1}s)\dif t.\]
        Denote by $\mathcal C_0$ the image of $C_c\big(G,C(h^uG)\big)$ under $\pi$.
        
        We prove that $\mathcal C$ commutes with $JAJ$ up to the compact operators. Since $C_c(G)$ is dense in 
        $\Cred(G)$ and $\mathcal C_0$ is dense in $\mathcal C$, it suffices to prove that
        for every $f\in C_c(G)$ and every $h\in C_c\big(G,C(h^u G)\big)$, we have 
        $T = [\pi(h),J\lambda(f)J]\in K(L^2(G))$. 
        A straightforward calculation 
        yields that for $\xi \in L^2(G)$ and $s\in G$, we have
        \begin{align*}
            (T\xi)(s)
              &= \int_G\int_G \big(h(t,s) - h(t,su)\big)\overline{f(u)} \delta_G(u)^{1/2}\xi(t^{-1}su)\dif t\dif u\\
        \end{align*}
        Let $(K_n)_n$ be an increasing sequence of compact subsets of 
        $G$ such that $G = \bigcup_n K_n$. Take a compact $L\subseteq G$ that contains the support of $f$ and of (the 
        first component of) $h$. Define the operator $T_n\in B(L^2(G))$ by
        \begin{align*}
            (T_n\xi)(s) 
               &= \int_G\int_G \chi_{K_n}(s)\big(h(t,s) - h(t,su)\big)\overline{f(u)} \delta_G(u)^{1/2}\xi(t^{-1}su)\dif t\dif u\\
               &= \int_G \int_G \chi_{K_n}(s) \big(h(t,s) - h(t,tu)\big) \overline{f(s^{-1}tu)} \delta_G(s^{-1}tu)^{1/2} \xi(u)\dif t\dif u\\
               &= \int_G k_n(s,u)\xi(u)\dif u
        \end{align*}
        where
        \[k_n(s,u) = \chi_{K_n}(s) \int_G \big(h(t,s) - h(t,tu)\big)\overline{f(s^{-1}tu)} \delta_G(s^{-1}tu)^{1/2}\dif t.\]
        Note that since $f\in C_c(G)$ and $h$ is compactly supported in the first component, we have that each
        $k_n\in L^2(G\times G)$ and hence that $T_n$ is compact. Moreover, $T_n\to T$ in norm 
        since
        \begin{align*}
            \norm{T\xi - T_n\xi}^2
             &= \int_{G\setminus K_n} \abs{ \int_G\int_G \big(h(t,s) - h(t,su)\big)\overline{f(u)} \delta_G(u)^{1/2}\xi(t^{-1}su)\dif t\dif u }^2 \dif s\\
             &\le \int_{G\setminus K_n} \sup_{t,u\in L} |h(t,s)-h(t,su)|^2\ \left(\int_L \int_L |f(u)|\;\delta_G(u)^{1/2}\; |\xi(t^{-1}su)|\dif u\dif t\right)^2\dif s\\
             &\le \sup_{s\in G\setminus K_n} \sup_{t,u\in L} |h(t,s)-h(t,su)|^2\ \mu(L)^2 \norm{J\lambda(|f|)J |\xi|}_2^2\\
             &= \sup_{s\in G\setminus K_n} \sup_{t,u\in L} |h(t,s)-h(t,su)|^2 \ \mu(L)^2 \norm{f}_1^2 \norm{\xi}_2^2
        \end{align*}
        and
        \[\limsup_{s\to \infty} |h(t,s)-h(t,su)|^2 = 0\]
        uniformly on compact sets for $t,u\in G$. We conclude that $T$ itself is compact.
    \end{proof}
    
    \printbibliography
\end{document}